\documentclass[11p,a4paper]{article}
\usepackage{amsfonts}
\usepackage{amssymb}
\usepackage{amsthm}
\usepackage{amsmath}
\usepackage{graphicx}
\usepackage{mathrsfs}
\usepackage{graphics}
\usepackage{abstract}
\usepackage{color}
\usepackage{cite}
\usepackage[shortlabels]{enumitem}
\usepackage{empheq}
\usepackage{appendix}
\usepackage{mathrsfs}
\usepackage{bm}
\usepackage[colorlinks=true,citecolor=red]{hyperref}
\hypersetup{linkcolor=blue}
\graphicspath{ {./result/} }
\usepackage{titling}
\usepackage{authblk}
\usepackage{geometry}
\usepackage{fancyhdr}
\usepackage{lineno}
\usepackage{titlesec}
\titleformat{\section}[block]{\centering\Large\bfseries}{\thesection}{1em}{}
\titleformat{\subsection}[block]{\centering\large\bfseries}{\thesubsection}{1em}{}
\titleformat{\subsubsection}[block]{\centering\normalsize\bfseries}{\thesubsubsection}{1em}{}

\usepackage{ulem}

\newtheorem{theorem}{\textbf{Theorem}}[section]
\newtheorem{lemma}{\textbf{Lemma}}[section]
\newtheorem{proposition}{\textbf{Proposition}}[section]
\newtheorem{corollary}{\textbf{Corollary}}[section]
\newtheorem{remark}{\textbf{Remark}}[section]
\newtheorem{definition}{\textbf{Definition}}[section]
\newtheorem{example}{\textbf{Example}}[section]

\providecommand{\keywords}[1]{\textbf{\textit{Keywords---}} #1}
\allowdisplaybreaks[4]

\numberwithin{equation}{section}
\def\be{\begin{equation}}
\def\ee{\end{equation}}
\def\bea{\begin{eqnarray}}
\def\eea{\end{eqnarray}}
\def\bt{\begin{theorem}}
\def\et{\end{theorem}}
\def\bl{\begin{lemma}}
\def\el{\end{lemma}}
\def\br{\begin{remark}}
\def\er{\end{remark}}
\def\bp{\begin{proposition}}
\def\ep{\end{proposition}}
\def\bc{\begin{corollary}}
\def\ec{\end{corollary}}
\def\bd{\begin{definition}}
\def\ed{\end{definition}}

\def\Pi{\mathbf{\psi}}

\def\R{{\mathbb R}}

\newcommand{\abs}[1]{\left\lvert #1 \right\rvert}
\newcommand{\norm}[1]{\left\| #1 \right\|}

\usepackage{tikz}
\usetikzlibrary{calc,arrows,intersections}
\title{\bf{On traveling wave solutions of water-wave equations in curved annular domains
}}

\author{
Liang Li$^{a}$
\thanks{liliang187@gxu.edu.cn}
,
Quan Wang$^{b}$
\thanks{Corresponding author:xihujunzi@scu.edu.cn }
\\ \footnotesize $^a$ School of Mathematics and Information Science, 
 Guangxi University, \\ \footnotesize
Nanning, Guangxi, 530004, P.R.China
  \\ \footnotesize $^{b}$ College of Mathematics, Sichuan University,
  \footnotesize
 Chengdu, Sichuan, 610065,  China
}
\date{}

\begin{document}
\maketitle

\begin{abstract}
This paper presents a pioneering investigation into the existence of traveling wave solutions for the two-dimensional Euler equations with constant vorticity in a curved annular domain, where gravity acts radially inward. This configuration is highly relevant to astrophysical and equatorial oceanic flows, such as planetary rings and equatorial currents. Unlike traditional water wave models that assume flat beds and vertical gravity, our study more accurately captures the centripetal effects and boundary-driven vorticity inherent in these complex systems.

Our main results establish both local and global bifurcation of traveling waves, marking a significant advancement in the field. First, through a local bifurcation analysis near a trivial solution, we identify a critical parameter \(\alpha_c\) and prove the existence of a smooth branch of small-amplitude solutions. The bifurcation is shown to be pitchfork-type, with its direction (subcritical or supercritical) determined by the sign of an explicit parameter \(\mathcal{O}\). This finding provides a nuanced understanding of the wave behavior under varying conditions.
Second, we obtain a global bifurcation result using a modified Leray-Schauder degree theory. This result demonstrates that the local branch extends to large-amplitude waves. This comprehensive analysis offers a holistic view of the traveling wave solutions in this complex domain. Finally, numerical examples illustrate the theoretical bifurcation types, confirming both supercritical and subcritical regimes. These examples highlight the practical applicability of our results.

This work represents the first rigorous analysis of traveling waves in curved domains with radial gravity and constant vorticity. By bridging mathematical theory with geophysical and astrophysical applications, our study not only fills a critical gap in the literature but also provides a robust framework for future research in this area.

\end{abstract}

\newpage
\tableofcontents

\section{Introduction}
The first comprehensive mathematical model to describe the motion of a fluid was proposed by Euler in the 1750's. The Euler's equations are as follows: 
\begin{align}\label{euler0812}
\begin{cases}
    \frac{\partial \rho}{\partial t}+\nabla\cdot\left(\rho\mathbf{u}\right)=0,~~x\in\Omega\subset R^{n},
    \\
    \frac{\partial \mathbf{u}}{\partial t}+\left(\mathbf{u}\cdot\nabla\right)\mathbf{u}+\frac{\nabla P}{\rho}=\mathbf{F},
\end{cases}
\end{align}
where $\Omega$ represents the fluid domain, with $\eqref{euler0812}_{1}$ and $\eqref{euler0812}_{2}$ expressing conservation of mass and conservation of momentum, respectively. $\mathbf{u}\left(x,t\right)$, $P\left(x,t\right)$ and $\rho\left(x,t\right)$ denote the velocity, pressure and density of the water at the spatial point $x$ and $t$, respectively, while $\rho\mathbf{F}$ accounts for external force like gravity. Remarkably, despite their 18th-century origin, the Euler's equations and their derived formulations continue to play a key role in fluid mechanics and engineering. 

We investigate the classical water wave problem, describing a water body bounded below by a rigid bed and above by a free surface interacting with the atmosphere, subject to gravitational forces. The mathematical study of water waves traces its origins to foundational works by Laplace (1776) and Lagrange (1781), who first derived linearized approximations about the quiescent state. Notably, Lagrange introduced particle following coordinates, now recognized as the Lagrangian formulation. Our introduction focuses on general traveling waves and solitary waves. General traveling waves are permanent-form solutions that propagate at constant velocity while preserving their waveform. Solitary waves constitute a special class of traveling waves characterized by localized, single-peak structures sustained through an exact balance of nonlinear and dispersive effects. Notably, traveling wave solutions may exhibit either periodic (e.g., Stokes waves) or non-periodic behavior; solitary waves are inherently non-periodic with a distinctive unimodal profile.

Traveling-wave solutions are characterized by solutions of the form $\mathbf{u}=\mathbf{u}\left(x-ct,y\right)$ in a 2-dimensional (2D) domain or $\mathbf{u}=\mathbf{u}\left(x-ct,y,z\right)$ in a 3D domain and $P=P\left(x-ct,y\right)$ in 2D or $P=P\left(x-ct,y,z\right)$ in 3D, where $c$ denotes the wave propagation speed. Typically, the density $\rho$ is assumed constant in these formulations. The mathematical study of water waves originated in the early 19th century with Cauchy and Poisson (1815), who employed Fourier expansions to analyze waves linearized about quiescent states. A pivotal advancement came with Airy's derivation of the dispersion relation $c^2k=g\tanh{kd}$ in 1845, where $d$ represents mean depth and $k$ the horizontal wavenumber. This Airy dispersion relation quantifies the fundamental frequency-wavenumber dependence of water waves, distinctly differentiating the dispersive behaviors of periodic waves (e.g., Stokes waves) and solitary waves--laying the groundwork for classifying wave types in subsequent research. For example, one can refer to \cite{Taklo2015, Armaroli2018}.

 The irrotational regime has attracted particular attention due to its mathematical tractability, dominating early research efforts. 
Stokes \cite {Stokes1847} studied irrotational periodic traveling water waves and
some of their nonlinear approximations. The rigorous construction of periodic water waves began with Nekrasov \cite {Nekrasov1921}, who first derived small-amplitude solutions for infinite depth through power series expansion in amplitude a, proving its convergence. Levi-Civita \cite {Levi-Civita1924} later simplified this approach without using Green's functions, while Strui k\cite {Struik1926} extended it to finite depth. These early results were limited to small amplitudes. The first large-amplitude periodic (irrotational) waves were constructed by Krasovskii \cite{Krasovskii1961} using Nekrasov's equation and Krasnoselskii's operator theory. Subsequent advances by Keady-Norbur y\cite{Keady1978} employed global bifurcation theory to obtain smooth solution curves, culminating in Amick-Fraenkel-Toland's \cite{Amick1982} construction of the extreme Stokes wave. More recently, Kogelbauer \cite {Kogelbauer2025} proved the variational instability for irrotational waves in finite-depth environments, expanding the understanding of their instability characteristics. Comprehensive treatments of irrotational wave theory can be found in \cite{Krasovskil1961, Toland1978, McLeod1997, Constantin2009}.

 Water waves with vorticity emerge in a diverse range of physical scenarios. For instance, in regions featuring distinct water density layers—such as the interface between warm surface waters and cold deep waters—internal waves form with embedded vorticity. These waves play a pivotal role in oceanic mixing, nutrient transport, and energy dissipation, thereby exerting significant influences on marine ecosystems and climate modeling frameworks \cite{Brizuela2023, Whitwell2024}. The theoretical exploration of traveling water waves with vorticity has witnessed a surge of interest after the groundbreaking work of Constantin and Strauss  \cite{Constantin2004}. To name a few, Constantin and Strauss \cite{Constantin2011} leverage elliptic theory to prove the existence of steady two-dimensional periodic water waves with large amplitudes in flows characterized by arbitrary bounded yet discontinuous vorticity distributions; Constantin and Escher \cite{Constantin2011a} establish that the profile of a periodic traveling wave propagating at the water surface above a flat bed, within a flow with real-analytic vorticity, must exhibit real-analyticity provided the wave speed exceeds the horizontal fluid velocity throughout the flow; Recent work \cite{Fei2024} by Xu et al. investigated two-dimensional steady stratified periodic gravity water waves, demonstrating that monotonic stratified waves of both large and small amplitude must exhibit crest-line symmetry when no stagnation points exist within the fluid domain. Furthermore, their analysis revealed that even when stagnation points occur away from the free surface, an analogous symmetry result for stratified water waves can be established through application of the moving plane method.  For more research on water waves, the interested readers can refer to \cite{Walsh20141, Constantin2016, Walsh20172, Constantin2019, Kristoffer2020, Constantin2021, Vladimir2023}.

We now turn our mathematical attention to solitary water waves - a specialized subclass of traveling wave solutions characterized by localized, non-periodic profiles that decay to a constant asymptotic height at infinity. Focusing on the fundamental case of two-dimensional irrotational waves in the absence of surface tension and density stratification, we survey the key developments in existence theory. The small-amplitude regime was first rigorously established by Lavrentiev \cite{Lavrent1952} through long-wavelength limits of weakly nonlinear periodic waves, with subsequent simplifications by Friedrichs and Hyers \cite{Friedrichs1954}. Later foundational works employed diverse analytical approaches: Beale \cite{Beale1977} utilized the Nash-Moser implicit function theorem, Mielke \cite{Mielke1988} implemented dynamical systems techniques via center manifold reduction, and Pego and Sun \cite{Pego2016} developed a modern fixed-point framework. The large-amplitude theory presents greater challenges, as standard global bifurcation methods fail to apply directly. Significant breakthroughs were achieved through Amick and Toland's \cite{Amick1981, Amick19811} approximation techniques and the alternative variational approach developed by Chen, Walsh, and their collaborators \cite{Wheeler2013, Wheeler2015, chen2016, chen2018}. These theoretical advances have been extended to more complex physical scenarios, including rotational flows\cite{Haziot2023, Matthies2024} and surface tension effects\cite{Dias2003, Groves2004}, demonstrating the rich mathematical structure underlying solitary wave phenomena.

In this paper, we consider the two-dimensional Euler equations with constant vorticity in a curved domain featuring a free surface (as illustrated by the shaded region in \autoref{quyu0525}).
\begin{align}\label{moxing0523}
\begin{cases}
    \rho\frac{\partial \mathbf{v}}{\partial t}
    +\rho\left(\mathbf{v}\cdot\nabla\right)\mathbf{v}=-\nabla P+\mathbf{g},~a<\sqrt{x^{2}+y^2}< \eta\left(t,\theta\right),
    \\
    \frac{\partial v_{1}}{\partial x}+\frac{\partial v_{2}}{\partial y}=0,~a<\sqrt{x^{2}+y^2}<\eta\left(t,\theta\right),
    \\
    \nabla\times\mathbf{v}=2w_{0},~a<\sqrt{x^{2}+y^2}<\eta\left(t,\theta\right),
\end{cases}
\end{align}
where \(\mathbf{v}=\left(v_{1},v_{2}\right)\) and \(P\) denotes the velocity and pressure field, respectively, \(\rho\) is the density of water.  Furthermore,
\(a>0\) is the radius of the bottom and \(\eta=\eta\left(t,\theta\right)\), representing the free surface, is a function depending on \(t\) and \(\theta\). \(\mathbf{g}=-\left(x,y\right)g/\sqrt{x^2+y^2}\) and the constant \(2\omega_{0}\) are the gravitational acceleration and vorticity, respectively.
\setcounter{equation}{0}
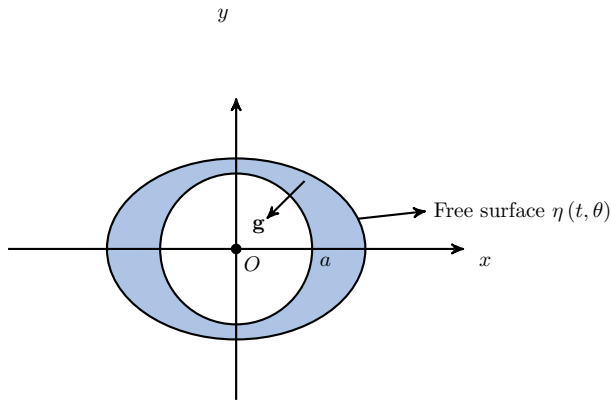
\begin{figure}[tbh]
\centering
        \begin{tikzpicture}[>=stealth',xscale=1,yscale=1,every node/.style={scale=0.8}]
\fill[color=cyan!50!blue!30] (0,0) ellipse (1.7 and 1.2); 
                      \fill[color=white] (0,0) circle (1);
\draw [thick,->] (0,-2) -- (0,2) ;
\draw [thick,->] (-3,0) -- (3,0) ;
\draw [thick,->] (0.9,0.9) -- (0.4,0.4) ;
\node [below right] at (3.1,0) {$x$};
\node [below right] at (1,0) {$a$};
\node at (0.3,0.3) {$\mathbf{g}$};
\node [left] at (0,3.1) {$y$};
 \draw[thick]  (0,0) circle (1); \draw[->] (-3,0) -- (3,0);
\draw[->] (0,-2) -- (0,2);
    \draw[thick,->] (0,0) ellipse (1.7 and 1.2);
    \draw[->,thick] (1.6,0.4) --(2.5,0.5);
    \node [right] at (2.5,0.5) {Free surface $\eta\left(t,\theta\right)$};
 \node[below right] at (0,0) {$O$};
     \fill (0,0) circle (2pt);
\end{tikzpicture}
\caption{$\mathbf{g}=-g\left(\frac{x}{r}, \frac{y}{r}\right)^T,\quad r=\sqrt{x^2+y^2}$}
\label{quyu0525}
\end{figure}

To our knowledge, existing research on water wave phenomena has predominantly focused on configurations with flat bottom topography and vertically oriented gravitational fields ($\mathbf{g} = (0,g)$). However, for realistic modeling of atmospheric and oceanographic flows—particularly in equatorial regions—our curved-domain formulation \eqref{moxing0523} provides a more physically accurate framework. This approach offers two significant improvements over conventional models: (i) the radially inward gravitational acceleration $\mathbf{g} = -\frac{(x,y)g}{r}$ precisely captures centripetal effects in celestial/equatorial flows, e.g., in planetary rings, it reproduces the balance between gravitational pull and rotational inertia, a key physical mechanism ignored by vertical $\mathbf{g}=(0,g)$ models. The work of Goldreich and Tremaine\cite{Goldreich1979} on planetary ring dynamics relies on similar radial gravity assumptions, highlighting the astrophysical relevance of our model; and (ii) the annular geometry naturally represents equatorial fluid motion (e.g., equatorial currents), inducing rotational behaviors like boundary-driven vorticity. In oceanography, our curved-domain, vorticity-aware model enriches equatorial wave studies, complementing Wang et al.\cite{Wang2025} who reconstructed Tropical Instability Waves (TIWs) in the equatorial Pacific. Their work on TIW dynamics, linked to equatorial current shear instability, aligns with our model's focus on annular geometry and vorticity effects, bridging theory and geophysical observations. However, from a mathematical perspective, it remains an open problem in applied mathematics whether the system \eqref{moxing0523} admits any traveling wave solutions.
\section{Reformulations}
Let $\Omega$ be the periodic channel $\mathbb{T} \times (-1, 0)$ or the rectangle $(0, 2\pi) \times (-1, 0)$ with periodicity $2\pi$ in the $q$-direction. We denote by $T$ the top boundary $T = \mathbb{T} \times \{p = 0\}$ of its closure $\overline{\Omega}$, and by $B$ the corresponding bottom boundary $B = \mathbb{T} \times \{p = -1\}$. The spaces $Y$ and $X$ are defined as follows:
\begin{align}\label{spaces}
\begin{aligned}
Y &= Y_1 \times Y_2, \quad 
Y_1 = C_{\text{per}}^{1+s}(\overline{\Omega}), \quad 
Y_2 = C_{\text{per}}^{2+s}(T), \\
X &= \left\{ h(q,p) \in C_{\text{per}}^{3+s}(\overline{\Omega}) \,\middle|\, 
\text{$h = 0$ on $B$, and $h$ is even and periodic in $q$} \right\},
\end{aligned}
\end{align}
where the subscript ``per'' indicates periodicity and evenness in the variable $q$.

\subsection{Equations in polar coordinates}

As the inner boundary is circular, it is naturally to reformulate the governing equations \eqref{moxing0523} for water waves in polar coordinates. Let \(x = r\cos\theta\) and \(y = r\sin\theta\). The transformation is given by
\begin{align}\label{bianhuan0618}
  u_{r} = v_{1}\cos\theta + v_{2}\sin\theta, \quad
  u_{\theta} = v_{2}\cos\theta - v_{1}\sin\theta.
\end{align}
Consequently, we obtain the following system:
\begin{align}\label{main-equations}
\begin{cases}
\frac{\partial u_r}{\partial t}+(u \cdot \nabla) u_r-\frac{u_\theta^2}{r}=-\frac{1}{\rho}\frac{\partial P}{\partial r}-g,\quad a<r<\eta(t,\theta),\\
\frac{\partial u_\theta}{\partial t}+(u \cdot \nabla)u_\theta+\frac{u_r u_\theta}{r}=-\frac{1}{r\rho}\frac{\partial P}{\partial \theta},\quad a<r<\eta(t,\theta),\\
\frac{\partial (r u_r)}{\partial r}
+\frac{\partial u_\theta}{\partial \theta}=0, \quad a<r<\eta(t,\theta),\\
\frac{\partial u_{\theta}}{\partial r}+\frac{u_{\theta}}{r}-\frac{1}{r}
\frac{\partial u_{r}}{\partial \theta}=2\omega_0,\quad a<r<\eta(t,\theta),
\end{cases}
\end{align}
where the operator $(u \cdot \nabla)$ is given by
\begin{align*}
(u \cdot \nabla)=u_r\frac{\partial }{\partial r}+\frac{u_\theta}{ r}\frac{\partial }{\partial \theta}.
 \end{align*}
The boundary conditions for the system \eqref{main-equations} are given as follows:
\begin{align}\label{bd-c}
\begin{cases}
\displaystyle
\frac{\partial \eta}{\partial t} + \frac{u_{\theta}}{r} \frac{\partial \eta}{\partial \theta} = u_{r}, & \quad r = \eta(t,\theta), \\
P = p_{\text{atm}}, & \quad r = \eta(t,\theta), \\
u_{r} = 0, & \quad r = a,
\end{cases}
\end{align}
where the first boundary condition represents the kinematic boundary condition at the free surface, ensuring that fluid particles on the surface remain there. The second condition is the dynamic boundary condition, which states that the pressure $P$ at the free surface equals the constant atmospheric pressure $p_{\text{atm}}$. The final condition imposes that the radial velocity $u_r$ vanishes at the inner boundary $r = a$, indicating that no fluid penetrates the inner wall.

\subsection{Non-dimensionalization}

To non-dimensionalize the problem described by equations \eqref{main-equations}--\eqref{bd-c}, we introduce the following characteristic scales and define the corresponding dimensionless variables:
\begin{align}\label{hai0819}
r = a\tilde{r}, \quad
t = \omega_0^{-1}\tilde{t}, \quad
(\eta, u_r, u_{\theta}, P) = \left(a\tilde{\eta},\ a\omega_0\tilde{u}_r,\ a\omega_0\tilde{u}_{\theta},\ a^2\omega_0^2\rho\tilde{Q}\right),
\end{align}
where $a$ denotes a characteristic length scale, $\omega_0$ is a characteristic frequency. The tilded variables $\tilde{r}$, $\tilde{t}$, $\tilde{\eta}$, $\tilde{u}_r$, $\tilde{u}_{\theta}$, and $\tilde{Q}$ are the corresponding dimensionless quantities.

Substituting these scaled variables into the original system  \eqref{main-equations}--\eqref{bd-c}, we derive the dimensionless form of the governing equations. After simplification and dropping the tildes for notational convenience, we obtain the following non-dimensional system:
\begin{align}\label{main-equations-f}
\begin{cases}
\frac{\partial u_r}{\partial t}+(u \cdot \nabla) u_r-\frac{u_\theta^2}{r}=-\frac{\partial Q}{\partial r}-\alpha,\quad 1<r<\eta(t,\theta),\\
\frac{\partial u_\theta}{\partial t}+(u \cdot \nabla)u_\theta+\frac{u_r u_\theta}{r}=-\frac{1}{r}\frac{\partial Q}{\partial \theta},\quad 1<r<\eta(t,\theta),\\
\frac{\partial (r u_r)}{\partial r}
+\frac{\partial u_\theta}{\partial \theta}=0, \quad 1<r<\eta(t,\theta),\\
\frac{\partial u_{\theta}}{\partial r}+\frac{u_{\theta}}{r}-\frac{1}{r}
\frac{\partial u_{r}}{\partial \theta}=2,\quad 1<r<\eta(t,\theta),
\end{cases}
\end{align}
where \(\alpha=g/(a w_{0}^2)\).
And the corresponding dimensionless boundary conditions are as follows,
\begin{align}\label{bd-cf}
\begin{cases}
\frac{\partial \eta}{\partial t}
+\frac{u_{\theta}}{r}\frac{\partial \eta}{\partial \theta}=u_r,\quad r=\eta(t,\theta),\\
Q=Q_{0},\quad r=\eta(t,\theta),\\
u_r=0,\quad r=1,
\end{cases}
\end{align}
where $Q_{0}$ is a dimenionless parameter given by
\[
Q_{0}=\frac{p_{atm}}{a^2\omega_0^2\rho}.
\]

We aim to look for the traveling wave solutions $\left(u_{r}, u_{\theta}, \eta, Q\right)$ to system \eqref{main-equations-f}--\eqref{bd-cf}. Specifically, we assume solutions of the form:
\begin{align}\label{traveling}
\begin{cases}
\eta(t,\theta) := S(\theta - t), \\
u_r(t,r,\theta) := U(r, \theta - t), \\
u_{\theta}(t,r,\theta) := V(r, \theta - t), \\
Q(t,r,\theta) := \Upsilon(r, \theta - t) + Q_{0},
\end{cases}
\end{align}
where $S$, $U$, $V$, and $\Upsilon$ are functions describing the wave profile in a co-moving frame. This ansatz eliminates the explicit time dependence and implies that the solutions appear steady and periodic in the $\theta$-direction when observed in a reference frame moving with constant unit speed to the right. The constant $Q_0$ represents a background or reference value of the variable $Q$.

Furthermore, we introduce the following change of variables:
\begin{align}\label{haishi0819}
\Theta = \theta - t, \quad R = r.
\end{align}
This transformation defines a coordinate system moving with the wave, where $\Theta$ represents the phase variable in the traveling frame and $R$ remains the radial coordinate. In these new coordinates, the partial derivatives transform as:
\[
\frac{\partial}{\partial t} = -\frac{\partial}{\partial \Theta}, \quad 
\frac{\partial}{\partial \theta} = \frac{\partial}{\partial \Theta}, \quad 
\frac{\partial}{\partial r} = \frac{\partial}{\partial R}.
\]
This change of variables simplifies the analysis by eliminating the explicit time dependence and reducing the problem to studying the spatial structure of the wave in the co-moving frame. Substituting the traveling wave ansatz \eqref{traveling} into the governing system \eqref{main-equations-f}--\eqref{bd-cf} and applying the coordinate transformation \eqref{haishi0819} yields the equations in the traveling wave coordinates:
\begin{align}\label{bianji1113}
\begin{cases}
\left(\frac{V}{R}-1\right)\frac{\partial U}{\partial \Theta}+U\frac{\partial U}{\partial R}-\frac{V^2}{R}=-
\frac{\partial \Upsilon}{\partial R}-\alpha,\quad 1<R<S(\Theta),\\
\left(\frac{V}{R}-1\right)\frac{\partial V}{\partial \Theta}+U\frac{\partial V}{\partial R}+\frac{UV}{R}=-\frac{1}{R}
\frac{\partial\Upsilon}{\partial \Theta},\quad 1<R<S(\Theta),\\
\frac{\partial (R U)}{\partial R}
+\frac{\partial V}{\partial \Theta}=0,\quad 1<R<S(\Theta),\\
\frac{\partial V}{\partial R}+\frac{V}{R}-
\frac{1}{R}\frac{\partial U}{\partial \Theta}=2,\quad 1<R<S(\Theta),
\end{cases}
\end{align}
and the boundary conditions \eqref{bd-cf} transform into:
\begin{align}\label{bd-2}
\begin{cases}
\left(V - R\right)S'(\Theta) = R\,U, & \quad R = S(\Theta), \\
\Upsilon = 0, & \quad R = S(\Theta), \\
U = 0, & \quad R = 1.
\end{cases}
\end{align}
Here, the first condition arises from the kinematic boundary condition at the free surface, the second condition corresponds to the dynamic boundary condition enforcing constant pressure at the surface, and the third condition expresses the no-penetration requirement.

Observe that the velocity field is divergence-free. This allows us to introduce a stream function \(\Psi = \Psi(R, \Theta)\) defined by:
\[
V - R = \frac{\partial \Psi}{\partial R}, \quad U = -\frac{1}{R} \frac{\partial \Psi}{\partial \Theta}.
\]
Using the identity
\[
\frac{\partial V}{\partial R} + \frac{V}{R} - \frac{1}{R} \frac{\partial U}{\partial \Theta} = 2,
\]
we find that \(\Psi\) satisfies the Lapalce equation:
\[
\left( \frac{\partial^2}{\partial R^2} + \frac{1}{R} \frac{\partial}{\partial R} + \frac{1}{R^2} \frac{\partial^2}{\partial \Theta^2} \right) \Psi = 0.
\]
Furthermore, from boundary conditions \(\eqref{bd-2}_1\) and \(\eqref{bd-2}_3\), the stream function \(\Psi\) also satisfies the following free boundary value problem:
\begin{align}\label{equation-stream}
\begin{cases}
\displaystyle
\left( \frac{\partial^2}{\partial R^2} + \frac{1}{R} \frac{\partial}{\partial R} + \frac{1}{R^2} \frac{\partial^2}{\partial \Theta^2} \right) \Psi = 0, & \quad 1 < R < S(\Theta), \\
\displaystyle
\frac{\partial \Psi}{\partial R} S'(\Theta) = -\frac{\partial \Psi}{\partial \Theta}, & \quad R = S(\Theta), \\
\displaystyle
\frac{\partial \Psi}{\partial \Theta} = 0, & \quad R = 1.
\end{cases}
\end{align}

The boundary conditions on the bottom \(R = 1\) and the free surface \(R = S(\Theta)\) imply that \(\Psi\) is constant on both boundaries.  
To verify this on the free surface, consider \(\Psi(R, \Theta)\) evaluated at \(R = S(\Theta)\):  
\[
\frac{d}{d\Theta} \Psi(S(\Theta), \Theta) = \frac{\partial \Psi}{\partial R} S'(\Theta) + \frac{\partial \Psi}{\partial \Theta}.
\]
Substituting the boundary condition \(\eqref{equation-stream}_2\), namely  
\[
\frac{\partial \Psi}{\partial R} S'(\Theta) = -\frac{\partial \Psi}{\partial \Theta},
\]
we obtain  
\[
\frac{d}{d\Theta} \Psi(S(\Theta), \Theta) = 0,
\]
which shows that \(\Psi\) is constant along the free surface. A similar argument applies at the bottom \(R = 1\), where the condition \(\frac{\partial \Psi}{\partial \Theta} = 0\) also implies that \(\Psi\) is constant.

We now define the quantity
\begin{align}\label{p0}
p_0 = \int_{1}^{S(\Theta)} \frac{\partial \Psi}{\partial R}  dR,
\end{align}
which is referred to as the relative mass flux. It follows directly that  
\[
p_0 = \Psi(S(\Theta), \Theta) - \Psi(1, \Theta)
\]  
is a constant, as can be verified using \(\eqref{equation-stream}_2\). Therefore, without loss of generality, we may take  
\[
\Psi = 0 \quad \text{on} \quad R = S(\Theta); \qquad \Psi = -p_0 \quad \text{on} \quad R = 1.
\]  
Introducing the scaling \(\Psi = -p_0 \tilde{\Psi}\) and omitting the tilde yields the following system:
\begin{align}\label{haishixuyao}
\begin{cases}
\displaystyle
\left( \frac{\partial^2}{\partial R^2} + \frac{1}{R} \frac{\partial}{\partial R} + \frac{1}{R^2} \frac{\partial^2}{\partial \Theta^2} \right) \Psi = 0, & \quad 1 < R < S(\Theta), \\
\displaystyle
\Psi = 0, & \quad R = S(\Theta), \\
\displaystyle
\Psi = 1, & \quad R = 1, \\
\displaystyle
V - R = -p_0 \frac{\partial \Psi}{\partial R}, \quad U = \frac{p_0}{R} \frac{\partial \Psi}{\partial \Theta}.
\end{cases}
\end{align}

We define the following quantities:
\begin{align}\label{Ber}
\begin{aligned}
E &:= \frac{\left(V - R\right)^2 + U^2}{2} + \alpha(R - 1) + \Upsilon+\frac{p_{atm}}{a^2\omega_0^2\rho} + 2p_{0}\Psi - \frac{R^2}{2}, \quad 1 < R < S(\Theta), \\
\lambda &:= \frac{\left(V - R\right)^2 + U^2}{2} + \alpha R - \frac{R^2}{2}, \quad R = S(\Theta).
\end{aligned}
\end{align}
From equation \eqref{bianji1113}, we find that the Bernoulli law throughout the flow can be derived.

\begin{lemma}\label{zong0426}
Let \(E\) and \(\lambda\) be defined by \eqref{Ber}. Then
$E$ and $\lambda$ are constants and
the following relations hold:
\begin{align}\label{changshu1110}
\lambda = E -\frac{p_{atm}}{a^2\omega_0^2\rho} + \alpha,~\text{on}~R = S(\Theta).
\end{align}

\end{lemma}

\begin{remark}
The proof of \autoref{zong0426} can be completed by performing a direct standard calculation using the governing equations \eqref{bianji1113}, conditions \eqref{bd-2} and the definitions of \(E\) and \(\lambda\).
\end{remark}

Hence, using \autoref{zong0426} and system \eqref{haishixuyao}, we derive the following reduced system:
\begin{align}\label{pde-m}
\begin{cases}
\left(\frac{\partial^2}{\partial R^2}+\frac{1}{R}
\frac{\partial }{\partial R}+\frac{1}{R^2}
\frac{\partial^2}{\partial \Theta^2}\right)
\Psi=0,\quad 1<R<S(\Theta),\\
\frac{\abs{p_{0}\nabla  \Psi}^2}{2}+\alpha (R-1)+\Upsilon+2p_{0}\Psi
-\frac{R^2}{2}=E,\quad 1<R<S(\Theta),\\
\abs{p_{0}\nabla  \Psi}^2-R^2+2\alpha R=2\lambda,\quad R=S(\Theta),\\
\Psi=0,\quad R=S(\Theta),\\
\Psi=1,\quad R=1.
\end{cases}
\end{align}
where the gradient operator \(\nabla\) in the \((R, \Theta)\) coordinates is defined by  
\[
\nabla \Psi = \left( \frac{\partial \Psi}{\partial R},\ -\frac{1}{R} \frac{\partial \Psi}{\partial \Theta} \right).
\]
The parameter \(\lambda\) is completely determined by \(\alpha\), \(E\), \(p_{\mathrm{atm}}\), \(a\), \(\omega_0\), and \(\rho\). In this article, we aim to establish the existence of nontrivial solutions to the system \eqref{pde-m} as \(\alpha\) varies.

 \subsection{Equivalent equations}
 
The physical fluid domain for system \eqref{pde-m} is unknown due to the fact that the free surface \(S(\Theta)\) is an undetermined function. It is therefore natural to transform the physical fluid domain into a fixed domain, and subsequently derive the corresponding system defined on this fixed domain. To this end, we focus on the case where \(V\) and \(p_0\) satisfy the following condition:
\[
(V - R) < 0 \quad \text{and} \quad p_0 < 0.
\]  
This condition guarantees that \(\Psi(R, \Theta)\) is strictly decreasing with respect to \(R\). We therefore introduce the following transformation:
\begin{align}\label{bianhuan1113}
q = \Theta, \quad p = -\Psi(R, \Theta).
\end{align}  

Due to the periodicity of the waves, it suffices to consider this transformation over one wavelength. Using \(\eqref{pde-m}_4\) and \(\eqref{pde-m}_5\), the transformed domain \(\Omega\) is found to be the periodic channel
\[
\Omega = \mathbb{T} \times (-1, 0).
\]  

Let \(h(q, p) = R - 1\) denote the height above the bottom. Applying the transformation \eqref{bianhuan1113} to \(h(q, p)\) and using the chain rule, we obtain:  
\begin{align}\label{suoyi1113}
\begin{aligned}
1 &= \frac{dh}{dR} = \frac{\partial h}{\partial p} \frac{\partial p}{\partial \Psi} \frac{\partial \Psi}{\partial R} = \frac{\partial h}{\partial p} \frac{V - R}{p_0}, \\
0 &= \frac{dh}{d\Theta} = \frac{\partial h}{\partial q} \frac{\partial q}{\partial \Theta} + \frac{\partial h}{\partial p} \frac{\partial p}{\partial \Psi} \frac{\partial \Psi}{\partial \Theta} = \frac{\partial h}{\partial q} - \frac{\partial h}{\partial p} \frac{R U}{p_0}.
\end{aligned}
\end{align}  
From these and \(\eqref{haishixuyao}_4\), it follows that  
\begin{align}\label{suoyi11131}
h_q := \frac{\partial h}{\partial q} = \frac{R U}{V - R}, \quad
h_p := \frac{\partial h}{\partial p} = \frac{p_0}{V - R}.
\end{align}  
Combining \(\eqref{haishixuyao}_4\) with \eqref{suoyi11131}, we derive  
\begin{align}\label{gaiyixia1113}
\frac{\partial \Psi}{\partial R} = -\frac{1}{h_p}, \quad
\frac{1}{R} \frac{\partial \Psi}{\partial \Theta} = \frac{1}{h + 1} \frac{h_q}{h_p}.
\end{align}

Similarly, applying the chain rule to the identity \(h(\Theta, -\Psi(R, \Theta)) = R - 1\), we further derive:
\begin{align}\label{zaituidaoyici1113}
\frac{\partial^2 \Psi}{\partial R^2} = \frac{h_{pp}}{(h_p)^3}, \quad
\frac{1}{R^2} \frac{\partial^2 \Psi}{\partial \Theta^2} = \frac{1}{(h+1)^2} \cdot \frac{h_{qq}(h_p)^2 - 2h_q h_p h_{qp} + h_{pp}(h_q)^2}{(h_p)^3}.
\end{align}
Substituting \eqref{gaiyixia1113}--\eqref{zaituidaoyici1113} into \eqref{pde-m}, we find that \(h(q, p)\) satisfies the following nonlinear system:
\begin{align}\label{new-equation}
\begin{cases}
\begin{aligned}
&h_{qq}(h_p)^2 - 2h_q h_p h_{qp} + h_{pp}(h_q)^2 -(h+1)(h_p)^2\\& +  (h+1)^2 h_{pp} = 0, \quad -1 < p < 0, 
\end{aligned}
\\
\begin{aligned}
&(h+1)^2 (h_p)^2 \left[ 2\lambda + (h+1)^2 - 2\alpha (h+1) \right]  \\&- p_0^2 (h_q)^2 -p_0^2 (h+1)^2 = 0, \quad p = 0,
\end{aligned}
 \\
\displaystyle
h = 0, ~ p = -1.
\end{cases}
\end{align}

\begin{lemma}\label{dengjiaxingzhengming0618} 
Under the conditions \(V - R < 0\) and \(p_0 < 0\), the existence of a classical solution to the problem \eqref{main-equations}--\eqref{bd-c} is equivalent to the existence of a solution to the problem \eqref{new-equation}.
\end{lemma}

\begin{remark}
The proof of \autoref{dengjiaxingzhengming0618}  is provided in \autoref{dengjiaxingzhengming0618-1}.
\end{remark}

 \subsection{Statement of main results}

For some \(\gamma\) satisfying \(0 < \gamma < 1\) , suppoe that \(p_0\)
satisfies 
\[
e^{4\gamma} \gamma^2 > p_0^2.
\]
We then define \(\alpha_s\) as follows:
\begin{align}
\alpha_s :=
 \frac{1}{2} \left( e^{\gamma} - \frac{p_0^2}{e^{3\gamma} \gamma^2} \right).
\end{align}
We then define the nonlinear operator  
\begin{align}\label{nonlinear-1}
\mathcal{G}(\alpha, h) = \left( \mathcal{G}_1(\alpha, h), \mathcal{G}_2(\alpha, h) \right) : (\alpha_s, +\infty) \times X \to Y = Y_1 \times Y_2,
\end{align}  
where the components are given by  
\begin{align}
&\begin{aligned}\label{nonlinear-e-1-1}
\mathcal{G}_1(\alpha, h) = {} & h_{qq} (h_p)^2 
 - 2h_q h_p h_{qp} + h_{pp}(h_q)^2 \\
&- (h+1)(h_p)^2 + (h+1)^2 h_{pp}
\end{aligned} \\
&\begin{aligned}\label{nonlinear-e-1-2}
\mathcal{G}_2(\alpha, h) = {} & (h+1)^2 (h_p)^2 \left( 2\lambda(\alpha, \gamma) + (h+1)^2 - 2\alpha (h+1) \right) \\
& - p_0^2 (h_q)^2 - p_0^2 (h+1)^2\big|_{p=0},
\end{aligned}
\end{align}
where $\lambda(\alpha, \gamma) $ is given by
\[
\lambda(\alpha, \gamma) := \frac{p_0^2 + e^{3\gamma} \gamma^2 \left(2\alpha - e^{\gamma}\right)}{2\gamma^2 e^{2\gamma}}>0\quad
\text{as}\quad e^{4\gamma} \gamma^2 > p_0^2.
\] 

 With the help of the nonlnear operator $\mathcal{G}(\alpha, h)$, the system
  \eqref{new-equation} is rewitten as 
  the following abstract operator equation defined on $(\alpha_s, +\infty) \times X$
  \begin{align}\label{operator-equation}
  \mathcal{G}(\alpha, h)=0.
  \end{align}
The main theorems of the paper are the following two existence results
on the system \eqref{operator-equation}.
\begin{theorem}\label{dingli08201} 
For some \(\gamma\) satisfying \(0 < \gamma < 1\), the following conclusions hold:
\begin{description}
\item[ (1)] The system \eqref{operator-equation} has a solution $H(p)= e^{\gamma(p + 1)} - 1$ for all $\alpha >\alpha_s$. It also admits a bifurcation point \((\alpha_c, H)\), where \(\alpha_c\) denotes the critical value of the control parameter $\alpha$, defined by
\[
\alpha_c = e^{\gamma} + \frac{2p_0^2}{\gamma^2 e^{\gamma}(e^{2\gamma} - 1)}.
\]
\item[ (2)]There exists a continuously differentiable nontrivial solution curve passing through the bifurcation point \((\alpha_c, H)\) and satisfying \(\mathcal{G}(\Lambda(s), x(s)) = \mathbf{0}\), parameterized as  
\[
\mathcal{C}_ {loc}:=\left\{ \left( \Lambda(s), x(s) \right) \middle| s \in (-\delta, \delta),\ \left( \Lambda(0), x(0) \right) = (\alpha_c, H) \right\}.
\] 
\item[ (3)] The bifurcation types of system \eqref{operator-equation} at $\left(\alpha_{c},H\right)$ is not transcritical, but determined by the sign of a parameter $\mathcal{O}$ given by
\[
\begin{aligned}
\mathcal{O}&=
\frac{1}{36 \gamma ^2 \hat{C}^2 p_{0}^2 \left(e^{2 \gamma }-1\right)^2 e^{9 \gamma}}\bigg{[}
60 p_{0}^4+230 p_{0}^4 e^{2 \gamma }+e^{4 \gamma } \left(962  p_{0}^4-79 \gamma^2  p_{0}^2+1152\gamma p_{0}^4\right)
\\
&\quad+e^{6 \gamma } \left(576 \gamma  p_{0}^4-3814 p_{0}^4+144 \gamma ^3 p_{0}^2+648 \gamma ^2 p_{0}^2\right)+e^{8 \gamma } (27 \gamma ^4+576 \gamma p_{0}^4\\
&\quad+1122 p_{0}^4-72 \gamma ^3 p_{0}^2
-1578 \gamma ^2 p_{0}^2)+e^{10 \gamma } \left(1440 p_{0}^4-72 \gamma ^4+988 \gamma ^2 p_{0}^2\right)\\
&\quad+e^{14 \gamma } \left(54 \gamma ^4-324 \gamma ^2 p_{0}^2\right)
+e^{12 \gamma } \left(36 \gamma ^4-72 \gamma ^3 p_{0}^2+345 \gamma ^2 p_{0}^2\right)
\\
&\quad-63 \gamma ^4 e^{16 \gamma }+18 \gamma ^4 e^{18 \gamma }
\bigg{]}.
\end{aligned}
\]
The bifurcation types is either subcritical if $\mathcal{O}<0$ or supercritical if $\mathcal{O}>0$,
 where
 \[
\hat{C}=\sqrt{\pi\left[\frac{1}{4\gamma}+\frac{3+4\gamma}{4\gamma e^{4\gamma}}-\frac{1}{\gamma e^{2\gamma}}\right]}.
 \]
\end{description}
\end{theorem}
\begin{remark}\label{chongxinlaiyoujihui0820}
The proof is presented in \autoref{433}, with its key steps outlined as follows:
\begin{itemize}
    \item The existence of \(H(p)=e^{\gamma(p + 1)} - 1\), the modified eigenvalue problem, and the critical control parameter \(\alpha_c\) are established in Theorem \ref{pingfanjie0820};
    \item  Define  \(\mathcal{F}(\alpha, f) := \mathcal{G}(\alpha, H + f)\).
    We then examine that the linear operator \(\mathcal{F}_{f}(\alpha_c, 0)\) satisfies the conditions of the Crandall–Rabinowitz Theorem (\autoref{the21}). Specifically, we verify whether both the dimension of the null space and the codimension of the image space equal one, as established in \autoref{yige0620}-\autoref{range0507}, respectively;
    \item To determine bifurcation types, we
    provide the computation of $\mathcal{O}$ in  \autoref{433}.
\end{itemize}

\end{remark}
 \autoref{dingli08201}  gives the local bifurcation conclusions about the systems \eqref{main-equations}-\eqref{bd-c} as follows.
\begin{corollary}[Small-amplitude waves] Consider traveling solutions \eqref{traveling} of speed
$-1$ and relative mass flux $p_0<0$ of the dimensionless water wave problem \eqref{main-equations}-\eqref{bd-c} with constant vorticity $2\omega_{0}$
 such that $u_{\theta}<ar\omega_{0}$ throughout the fluid. Under the assumptions of the conditions in \autoref{dingli08201}, there exists a connected set of solutions to the systems \eqref{main-equations}-\eqref{bd-c} bifurcated from $(a_{c},(u_{r0},u_{\theta 0},P_{0},\eta_{0}))$, where $a_{c}$ is a constant and $(u_{r},u_{\theta},P,\eta)=(u_{r0},u_{\theta 0},P_{0},\eta_{0})$ is a solution of \eqref{main-equations}-\eqref{bd-c}
 corresponding to the solution $H$ to the system \eqref{operator-equation}.
\end{corollary}
\begin{remark}
The absence of detailed information on the local bifurcation structure of the systems \eqref{main-equations}--\eqref{bd-c} stems from the implicit relationship between \(h\) and \(\Psi\), which complicates the explicit characterization of bifurcation behavior.
\end{remark}

\autoref{dingli08201} establishes the local bifurcation of system \eqref{operator-equation} at \((\alpha_c, H)\). The bifurcated solution near \((\alpha_c, H)\) corresponds to small-amplitude traveling waves. To demonstrate the existence of large-amplitude traveling waves, we must analyze the solution component containing the point \((\alpha_c, H(p))\). Our second result is a global bifurcation theorem for system \eqref{new-equation}.

To state global bifurcation theorem, for \(\gamma > \delta > 0\), define the set \(\mathcal{O}_{\delta}\) as follows:
\begin{align}\label{mathcal{O}}
\mathcal{O}_{\delta} = 
\left\{
(\alpha, h) \in (\alpha_s, +\infty) \times X \,\middle|\, 
(h + 1) h_p > \delta ~ \text{in} ~ \overline{D}
\right\} \cap \mathcal{P}_{\delta},
\end{align}
where \(\mathcal{P}_{\delta}\) is given by
\begin{align*}
\mathcal{P}_{\delta} = 
\left\{
(\alpha, h) \in (\alpha_s, +\infty) \times X \,\middle|\, 
(h + 1)^3 + 2\lambda(\alpha, \gamma)(h + 1) - 2\alpha (h + 1)^2 > \delta ~ \text{on} ~ T 
\right\},
\end{align*}
We then define  
\[
\mathcal{Q}_{\delta} = \overline{ \{ (\alpha, h) \in \mathcal{O}_{\delta} : \mathcal{G}(\alpha, h) = 0,~ h_q \neq 0 \} }\subset (\alpha_s, \infty) \times X,
\]  
where  \(X\) is defined in \eqref{spaces}. Let \(\mathcal{C}_{\delta}\) be the connected component of \(\mathcal{Q}_{\delta}\) that contains the point \((\alpha_c, H(p))\). Thus, \(\mathcal{C}_{\delta}\) includes \(\mathcal{C}_{\text{loc}}\).
For \(\mathcal{C}_{\delta}\) , we have the followng theorem:
\begin{theorem}[Large-amplitude traveling-wave solution]  \label{large-waves}
For each $\delta \in (0,\gamma)$. Either
\begin{description}
\item[ (1)]  $\mathcal{C}_{\delta}$  is unbounded in $ (\alpha_s,\infty) \times X$, or
\item[ (2)] $\mathcal{C}_{\delta}$  contains another trivial point $(\alpha_k, H)$, or
\item[ (3)]  $\mathcal{C}_{\delta}$ contains a point $(\alpha_k, h) \in \partial \mathcal{O}_ {\delta}$ ,
\end{description}
where $\alpha_k$ is determined by the following equation
\begin{align}\label{new-equation-per-ode-3}
 \begin{cases}
\partial_p((H+1)^{-2}M_{p})+\gamma^2(1-k^2)(H+1)^{-2}M=0,\\
\beta(\alpha_k,\gamma)  M-\gamma^{-1}M_p=0,\quad p=0,\\
M=0,\quad p=-1,
\end{cases}
\end{align}
which has nontrivial solution, and in which 
\begin{align}\label{beta}
\beta(\alpha, \gamma) := \gamma^2 e^{3\gamma} \left( \alpha - e^{\gamma} \right).
\end{align}
\end{theorem}
\subsection{Main technical difficulties behind the proof}

For the system \eqref{moxing0523}, the existence of traveling wave solutions is equivalent to the existence of solutions to the free boundary problem \eqref{equation-stream}. However, since \(S(\Theta)\) is an unknown function to be determined, we cannot directly apply bifurcation theory to study the bifurcation behavior of its solutions, as the parameter governing the bifurcation is not explicitly identified.

To address this challenge, we observe that \(p_0 = \int_{1}^{S(\Theta)} \frac{\partial \Psi}{\partial R}  dR\) is a constant. Combined with the assumption \(V < R\), this allows us to use the transformation \eqref{bianhuan1113} to convert problem \eqref{equation-stream} into problem \eqref{new-equation}. A careful analysis shows that the two problems are equivalent. On one hand, problem \eqref{new-equation} eliminates the free surface. On the other hand, by seeking the \(p\)-dependent solution \(H(p)\) to problem \eqref{new-equation} (see \autoref{pingfanjie0820}), we identify the key control parameter \(\alpha\) that determines the bifurcation behavior of the solutions.

Regarding local bifurcation, there is relatively little literature applying rigorous mathematical theory to classify bifurcation types in water wave problems, primarily due to the extensive computations required.  We computed \(\mathcal{F}_{f}(\alpha_c, f)\), \(\mathcal{F}_{ff}(\alpha_c, f)\), and \(\mathcal{F}_{fff}(\alpha_c, f)\). Since our results indicate that the bifurcation is not transcritical, it was necessary to solve the partial differential equation \eqref{buweiyi0713}. Although the general solution consists of a particular solution plus the solution to the homogeneous equation (obtainable via separation of variables), the computation of the particular solution posed significant challenges. During implementation, we streamlined the process by leveraging the structure of the right-hand side, ultimately deriving the expression for \(\mathcal{O}\) in \eqref{fuhao0715}.

For global bifurcation, the number of positive eigenvalues of the problem
\[
\begin{cases}
\mathcal{G}_{1w}(\alpha, w)h = \sigma h, \\
\mathcal{G}_{2w}(\alpha, w)h = \sigma h
\end{cases}
\]
prevented the direct application of classical topological theory to \eqref{operator-equation}. We instead show that the number of positive eigenvalues of the modified problem
\[
\begin{cases}
\mathcal{G}_{1w}(\alpha, w)h = \sigma h, \\
\mathcal{G}_{2w}(\alpha, w)h = 0
\end{cases}
\]
is finite. This result, combined with \autoref{lemma31}–\autoref{lemma32}, allows us to extend classical topological theory to problem \eqref{operator-equation} (see  \autoref{degree-theory}).

The remainder of this paper is organized as follows: Section \ref{zhengming0820} presents the proof of the local bifurcation theorem; Section \ref{formu0520} details the global bifurcation theorem; and Section \ref{shuzhi0805} provides a numerical investigation of the local bifurcation.

\section{Proofs of \autoref{dengjiaxingzhengming0618}}\label{dengjiaxingzhengming0618-1}


\begin{proof}
Assume that the problem \eqref{main-equations}--\eqref{bd-c} admits a solution \(\left(u_{r}, u_{\theta}, P, \eta\right)\). Then, by following the derivation steps from \eqref{main-equations}--\eqref{bd-c} to \eqref{new-equation}, one can show that there exists a function \(h(p, q)\) satisfying problem \eqref{new-equation}. Here, all functions involved are periodic in the variables \(q\).

Conversely, suppose \(h = h(q, p) \in C^2(\overline{\Omega})\) with \(h_p = \frac{\partial h}{\partial p} > 0\) is a solution to problem \eqref{new-equation}.
Let us define two functions $F(q, p)$ and $G(q, p)$ through $h(q, p)$ as follows
\begin{align*}
F(q, p) = \frac{1}{h_p}, \quad G(q, p) = -\frac{h_q}{h_p}.
\end{align*}  
Using the equality \(h_{qp} = h_{pq}\), it follows that  
\begin{align}\label{nuli0426}
F_q + F_p G - G_p F = 0.
\end{align}  
The free surface of the flow is given by \(S(\Theta) = h(\Theta, 0) + 1\).

For any fixed \(\Theta \in (0, 2\pi)\), consider the ordinary differential equation  
\begin{align}\label{jie0426}
\begin{cases}
\displaystyle
\frac{\partial \Psi}{\partial R}(R, \Theta) = -F(\Theta, -\Psi(R, \Theta)), \\
\Psi(S(\Theta), \Theta) = 0.
\end{cases}
\end{align}  
By the smoothness of \(F\), it follows that \eqref{jie0426} admits a unique local solution \(\Psi(R, \Theta)\), which is strictly decreasing in \(R\). 
Moreover, since \(F \geq \delta > 0\) for some constant \(\delta\), the function \(\Psi(R, \Theta)\) increases at a rate no less than \(\delta\) as \(R\) decreases. 
This allows the solution to be extended to a value \(y = y(\Theta)\) for each fixed \(\Theta\) such that  
\[
\Psi(y(\Theta), \Theta) = 1 \quad \text{and} \quad y(\Theta) < S(\Theta).
\]  
Due to the periodicity of \(F(\Theta, p)\) in \(\Theta\), the solution \(\Psi(R, \Theta)\) is also periodic in \(\Theta\).

Subsequently, we prove that  
\begin{align}\label{nu0426}
y(\Theta) = 1, \quad \forall ~\Theta \in (0, 2\pi).
\end{align}  
To this end, we first show that  
\begin{align}\label{diyibu0426}
\Psi_\Theta(R, \Theta) = -G(\Theta, -\Psi(R, \Theta)), \quad \text{for all } R \in [y(\Theta), S(\Theta)].
\end{align}  
For convenience, let us define \(W(\Theta, R) = -G(\Theta, -\Psi(R, \Theta))\). A direct computation yields:  
\begin{align*}
\begin{cases}
W_R = G_p \Psi_R = -G_p F = -F_q - F_p G = -F_q + F_p W, \\
(\Psi_\Theta)_R = -F_q + F_p \Psi_\Theta,
\end{cases}
\end{align*}  
where we have used \eqref{nuli0426} and \eqref{jie0426}. Thus, \(\Psi_\Theta\) and \(W\) satisfy the same equation.  Next, we show that they share the same initial data. It follows from \(\Psi(S(\Theta), \Theta) = 0\) and \(h(\Theta, 0) = S(\Theta) - 1\) 
that 
\[
\Psi_\Theta(S(\Theta), \Theta) = F(\Theta, 0) S'(\Theta), \quad h_q(\Theta, 0) = S'(\Theta).
\]
Furthermore, we have  
\[
\begin{aligned}
W(\Theta, S(\Theta)) &= -G(\Theta, -\Psi(S(\Theta), \Theta)) \\
&= h_q(\Theta, 0) F(\Theta, -\Psi(S(\Theta), \Theta)) \\
&= F(\Theta, 0) S'(\Theta),
\end{aligned}
\]  
where we have used the fact that \(\Psi(S(\Theta), \Theta) = 0\). Thus, we obtain  
\[
W(\Theta, S(\Theta)) = \Psi_\Theta(S(\Theta), \Theta),
\]  
which shows that \(\Psi_\Theta\) and \(W\) satisfy the same initial condition at \(R = S(\Theta)\).  Therefore, by the uniqueness theorem for ordinary differential equations, we conclude that \eqref{diyibu0426} holds.

Moreover, from \(\Psi(y(\Theta), \Theta) = 1\), we differentiate with respect to \(\Theta\) to obtain  
\[
\Psi_R(y(\Theta), \Theta) y'(\Theta) + \Psi_\Theta(y(\Theta), \Theta) = 0.
\]  
Substituting the expressions for \(\Psi_R\) and \(\Psi_\Theta\) from \eqref{jie0426} and \eqref{diyibu0426}, respectively, yields  
\[
-F(\Theta, -1) y'(\Theta) - G(\Theta, -1) = 0.
\]  
Since \(h(\Theta, -1) = 0\), it follows that \(h_q(\Theta, -1) = 0\), which implies \(G(\Theta, -1) = 0\). Additionally, \(F(\Theta, -1) = \frac{1}{h_p(\Theta, -1)} > 0\). Therefore, $y'(\Theta) = 0$,
which shows that \(y(\Theta)\) is constant.  

To determine the value of \(y\), we use the conditions \(h(\Theta, -1) = 0\) and \(h(0, 0) = S(0) - 1\). Applying the Newton–Leibniz formula, we obtain  
\begin{align*}
\begin{aligned}
S\left(0\right)-1
&=h\left(0,0\right)
=\int_{-1}^{0}h_{p}\left(0,p\right)dp=\int^{0}_{-1}\frac{1}{F\left(0,p\right)}dp 
\\
&=\int_{y\left(0\right)}^{S\left(0\right)}\frac{1}{F\left(0,-\Psi\left(R,0\right)\right)}
d\left(-\Psi\left(R,0\right)\right)
\\&=\int_{y\left(0\right)}^{S\left(0\right)}\frac{1}{-\Psi_{R}\left(R,0\right)}
d\left(-\Psi\left(R,0\right)\right)
\\
&=S\left(0\right)-y\left(0\right),
\end{aligned}
\end{align*}
which implies that \(y\left(0\right)=1\). Then \eqref{nu0426} is verified.
We now have  
\[
\Psi(1, \Theta) = 1 \quad \text{and} \quad \Psi(S(\Theta), \Theta) = 0.
\]  

We now prove the folloiwng identity  
\begin{align}\label{yiguzuoqi0426}
h(\Theta, -\Psi(R, \Theta)) = R - 1.
\end{align}  
Differentiating \(h(\Theta, -\Psi(R, \Theta))\) with respect to \(R\) and \(\Theta\), and using \eqref{jie0426} and \eqref{diyibu0426}, we obtain  
\[
\frac{\partial}{\partial R} h(\Theta, -\Psi(R, \Theta)) = 1, \quad
\frac{\partial}{\partial \Theta} h(\Theta, -\Psi(R, \Theta)) = 0.
\]  
Combined with \(\eqref{new-equation}_3\), namely \(h(\Theta, -1) = 0\), these derivatives imply that \eqref{yiguzuoqi0426} holds.

Subsequently, following the steps in \eqref{zaituidaoyici1113} to derive \eqref{pde-m}, we obtain \(\eqref{pde-m}_1\) and \(\eqref{pde-m}_3\) from \(\eqref{new-equation}_1\) and \(\eqref{new-equation}_2\), respectively. The equation \(\eqref{pde-m}_2\) follows from \autoref{zong0426}. Thus, we conclude that \(\Psi(R, \Theta)\) is a solution to the system \eqref{pde-m}.   Using the scaling relations \eqref{hai0819}, the coordinate transformation \eqref{haishi0819}, and the velocity expressions in \(\eqref{haishixuyao}_4\), we define  
\[
u_r = \frac{a^2 \omega_0 p_0}{r} \frac{\partial \Psi}{\partial \Theta} \Big|_{\Theta = \theta - t \omega_0}, \quad
u_\theta = r \omega_0 - a \omega_0 p_0 \frac{\partial \Psi}{\partial R} \Big|_{R = \frac{r}{a}}, \quad
\eta = a h(\theta - t \omega_0, 0) + a.
\]  
The pressure \(P\) is then determined by solving \(\eqref{main-equations}_1\) and \(\eqref{main-equations}_2\). Therefore, we obtain a solution 
to the system \eqref{main-equations}--\eqref{bd-c} satisfying
\[
(u_r, u_\theta, \eta, P) \in C^1(\overline{\Omega}) \times C^1(\overline{\Omega}) \times C^1(0, 2\pi) \times C^1(\overline{\Omega})
\]  

\end{proof}

\section{Small-amplitude traveling waves }\label{zhengming0820}
This section is devoted to considering the existence of local bifurcation-small-amplitude waves. As described in \autoref{chongxinlaiyoujihui0820}, we will divide several steps to obtain \autoref{dingli08201} .
\subsection{Bifurcation point}\label{pingfanjie0820}
To investigate the bifurcation structure of the system \eqref{new-equation}, it is essential to identify a particular class of solutions. For this purpose, we consider solutions that depend solely on the variable \(p\), i.e.,
\[
h(q, p) = H(p),
\]  
where \(H := H(p)\) is a function independent of \(q\). Under this assumption, the system \eqref{new-equation} reduces to an ordinary differential equation for \(H(p)\)
as follows
\begin{align}\label{new-equation-tivial}
 \begin{cases}
(H+1)H''-(H')^2=0,\\
(H')^2\left[
2\lambda+(H+1)^2-2\alpha (H+1)\right]-p_{0}^2
=0,\quad p=0,\\
H=0,\quad p=-1.
\end{cases}
\end{align}

One can solve the simplified system \eqref{new-equation-tivial} explicitly and obtain the solution  
\begin{align}\label{one-s}
H(p) = e^{\gamma(p + 1)} - 1,
\end{align}  
where the parameter \(\gamma\) satisfies the algebraic equation  
\begin{align}\label{biaodaishi0821}
 p_0^2 + \gamma^2 e^{2\gamma} \left(2\alpha e^{\gamma} - 2\lambda - e^{2\gamma}\right) = 0.
\end{align}
Moreover, we explicitly have  
\[
\lambda(\alpha, \gamma) := \frac{p_0^2 + e^{3\gamma} \gamma^2 \left(2\alpha - e^{\gamma}\right)}{2\gamma^2 e^{2\gamma}}.
\]  
Here, we consistently assume the following conditions:  
\[
0 < \gamma < 1, \quad e^{4\gamma} \gamma^2 > p_0^2, \quad \text{and} \quad \alpha \geq \alpha_s := \frac{1}{2} \left( e^{\gamma} - \frac{p_0^2}{e^{3\gamma} \gamma^2} \right),
\]  
which together ensure that  
\[
\alpha > 0 \quad \text{and} \quad \lambda(\alpha, \gamma) \geq 0.
\]

In the subsequent analysis, we treat \(\alpha\) as the principal control parameter for the system \eqref{new-equation}. We will demonstrate that there exists a critical value \(\alpha_c\) such that for \(\alpha > \alpha_c\), the system \eqref{new-equation} admits a solution \(h(q, p)\) bifurcating from the trivial solution \eqref{one-s}. This bifurcated solution corresponds to a traveling wave solution of the governing water wave equations \eqref{main-equations}.  

To establish this result, we linearize the nonlinear system \eqref{operator-equation} about the base solution \(H(p)\) and obtain the following eigenvalue problem:
\[
D_h \mathcal{G} \big|_{h = H} w = \sigma w,\quad w\in X,
\]
where the Fréchet derivative \(D_h \mathcal{G} \big|_{h = H}\) is explicitly given by the following expressions:
\begin{align}\label{new-equation-per1}
 \begin{cases}
(H+1)^2(\gamma^2w_{qq}+w_{pp})-2(H+1)H_pw_p+\left[2(H+1)H_{pp}-(H_p)^2\right]w=\sigma w,\\
2e^{\gamma}\left[p_{0}^2\gamma^{-1} w_p-\beta(\alpha,\gamma)w \right] =\sigma w,\quad p=0,\\
w=0,\quad p=-1.
\end{cases}
\end{align}
The preceding eigenvalue problem is equivalent to
\begin{align}\label{new-equation-per}
 \begin{cases}
\partial_q\left[(H+1)^{-2}\gamma^2w_{q}\right]+\partial_p\left[(H+1)^{-2}w_{p}\right]+\gamma^2(H+1)^{-2}w=\sigma  (H+1)^{-4}w,\\
2e^{\gamma}\left[p_{0}^2\gamma^{-1} w_p-\beta(\alpha,\gamma)w \right] =\sigma w,\quad p=0,\\
w=0,\quad p=-1.
\end{cases}
\end{align}
where the function \(\beta(\alpha, \gamma)\), depending on the parameters \(\gamma\) and \(\alpha\), is defined by  \eqref{beta}.

For the equation \eqref{new-equation-per}, the periodic boundary conditions allow us to seek a solution of the form  
\[
w(q, p) = M(p) \cos(kq),
\]  
where \(k \in \mathbb{Z}\) is the wave number. Substituting this ansatz into \eqref{new-equation-per}, we find that \(M(p)\) satisfies the following ordinary differential eigenvalue problem:  
\begin{align}\label{new-equation-per-ode}
\begin{cases}
\displaystyle
\partial_p \left[ (H + 1)^{-2} M_p \right] + \gamma^2 (1 - k^2) (H + 1)^{-2} M = \sigma (H + 1)^{-4} M, & -1 < p < 0, \\
\displaystyle
2e^{\gamma} \left[ p_0^2 \gamma^{-1} M_p - \beta(\alpha, \gamma) M \right] = \sigma M, & p = 0, \\
\displaystyle
M = 0, & p = -1.
\end{cases}
\end{align}

Our objective is to determine the smallest critical value \(\alpha_c\) of the parameter \(\alpha\) such that the eigenvalue \(\sigma\) of the preceding problem vanishes for some nonzero wave number \(k \neq 0\), i.e.,  
\[
\sigma(\alpha_c) = 0.
\]  
To this end, we introduce the following modified eigenvalue problem:
\begin{align}\label{new-equation-per-ode-212}
 \begin{cases}
\partial_p\left[(H+1)^{-2}M_{p}\right]+\gamma^2(1-k^2)(H+1)^{-2}M=\tilde{\sigma}  (H+1)^{-4}M,\\
2e^{\gamma}\left[p_{0}^2\gamma^{-1}M_p-\beta(\alpha,\gamma) M \right]=0,\quad p=0,\\
M=0,\quad p=-1.
\end{cases}
\end{align}

Note that the smallest value \(\alpha_c\) of \(\alpha\) for which the modified eigenvalue problem satisfies \(\tilde{\sigma}(\alpha_c) = 0\) with \(k \neq 0\) is identical to that of the original eigenvalue problem \eqref{new-equation-per-ode}. Therefore, to determine \(\alpha_c\), it suffices to analyze the modified eigenvalue problem \eqref{new-equation-per-ode-212}.  

This modified eigenvalue problem possesses a variational structure, enabling the definition  
\[
-\tilde{\sigma}_k(\alpha) := \inf_{M \in \mathcal{H}} \frac{ 
E_1(\alpha,M)
}{ \int_{-1}^0 (H + 1)^{-4} M^2  dp },
\]  
where the functional $E_1(\alpha,M)$ is defined by
\[
\begin{aligned}
E_1(\alpha,M)= & -\beta(\alpha, \gamma) \gamma e^{-2\gamma} M^2(0) / p_0^2 
+ \gamma^2 (k^2 - 1) \int_{-1}^0 (H + 1)^{-2} M^2  dp 
    \\&+ \int_{-1}^0 (H + 1)^{-2} M_p^2  dp,
\end{aligned}
\]
and \(\mathcal{H}\) denotes an appropriate function space incorporating the boundary conditions, defined by 
\[
\mathcal{H}=\left\{M\in H^{1}\left(\left[-1,0\right]\right)|M(-1)=0\right\},
\]
from which one can see that $\alpha_c$ is determined by $\tilde{\sigma}_1(\alpha)=0$. Taking $k=1$, we then have
\begin{align}\label{eigen-variation}
-\tilde{\sigma}_1(\alpha):=\inf_{M\in \mathcal{H}}\frac{-\beta(\alpha,\gamma) \gamma
e^{-2\gamma}
M^2(0)/p_{0}^2
+
\int_{-1}^0(H+1)^{-2}M_{p}^2\,dp}{\int_{-1}^0(H+1)^{-4}M^2\,dp}.
\end{align}

Let us choose the test function
\[
M(p)=e^{\gamma(p+1)}-1
\]
Then, we obtain the upper bound  
\begin{align*}
-\tilde{\sigma}_1(\alpha)\leq \eta(\alpha):&=
\frac{\gamma^3e^{\gamma }
\left(
e^{\gamma }-\alpha 
\right)
M^2(0)/p_{0}^2
+
\int_{-1}^0(H+1)^{-2}M_{p}^2\,dp}{\int_{-1}^0(H+1)^{-4}M^2\,dp}
\\
&=\frac{\gamma^2\left[\gamma e^{\gamma}\left(e^\gamma-\alpha\right)\left(e^\gamma-1\right)^2/p_{0}^2+1\right]}{\int_{-1}^0(H+1)^{-4}M^2\,dp}.
\end{align*}
From this, it can be observed that there exists a value \(\alpha_0\) such that  
\[
-\tilde{\sigma}_1(\alpha) \leq \eta(\alpha) < 0, \quad \text{for all } \alpha > \alpha_0.
\]  
Note that for \(\alpha \leq e^{\gamma}\), it follows from \eqref{eigen-variation} that  
\[
-\tilde{\sigma}_1(\alpha) > 0.
\]  
Since \(-\tilde{\sigma}_1(\alpha)\) depends continuously on \(\alpha\), there must exist a critical value \(\alpha_c\) such that  
\begin{align}\label{guji1207}
\tilde{\sigma}_1(\alpha_c) = 0.
\end{align}
The critical value $\alpha_c$ is determined by the existence of nontrivial solution of 
the following equation
\begin{align}\label{new-equation-per-ode-2}
 \begin{cases}
\partial_p\left[(H+1)^{-2}M_{p}\right]=0,\\
\gamma^2e^{3\gamma }
\left(
\alpha_c -e^{\gamma }
\right) M-p_{0}^2\gamma^{-1}M_p=0,\quad p=0,\\
M=0,\quad p=-1.
\end{cases}
\end{align}
The general solution of the differential equation \(\partial_p \left[ (H + 1)^{-2} M_p \right] = 0\) is given by  
\[
M(p) = B_1 e^{2\gamma(p + 1)} + B_2,
\]  
where \(B_1\) and \(B_2\) are arbitrary constants. From this, it follows that the boundary value problem \eqref{new-equation-per-ode-2} admits a nontrivial solution if and only if  
\begin{align}\label{critical-alpha-1}
\begin{vmatrix}
1 & 1 \\
e^{2\gamma} \beta(\alpha_c, \gamma) - 2p_0^2 e^{2\gamma} & \beta(\alpha_c, \gamma)
\end{vmatrix} = 0.
\end{align}
We then obtain the critical value  
\begin{align}\label{critical-alpha}
\alpha_c = e^{\gamma} + \frac{2p_0^2}{\gamma^2 e^{\gamma} (e^{2\gamma} - 1)},
\end{align}  
which is clearly greater than \(\alpha_s = \frac{1}{2} \left( e^{\gamma} - \frac{p_0^2}{e^{3\gamma} \gamma^2} \right)\).  

Furthermore, it can be readily shown that  
\begin{align}\label{guji12071}
&\begin{cases}
\tilde{\sigma}_0(\alpha_c) > 0, \\
\tilde{\sigma}_k(\alpha_c) < 0, \quad k \geq 2,
\end{cases}\\
\label{jizhu0805}
&\sigma_1(\alpha) 
\begin{cases}
> 0, & \alpha > \alpha_c, \\
= 0, & \alpha = \alpha_c, \\
< 0, & \alpha < \alpha_c.
\end{cases}
\end{align}  
Hence, \((\alpha_c, H)\) is a bifurcation point, a fact that will be established in the following subsection.

\subsection{Existence of bifurcation solution}\label{formu0820}

The Crandall-Rabinowitz theorem \cite{Kielhoefer2012} on bifurcation from a simple eigenvalue is a fundamental tool for studying bifurcation in nonlinear equations. 
We also use this theorem to study the problem \eqref{new-equation}.
For convenience, let us recall the Crandall-Rabinowitz theorem.
\begin{theorem}[Crandall-Rabinowitz]\label{the21}
Let $X$ and $Y$ be two Banach spaces, $I$ be an interval in $\R$ containing 
$\alpha_c$, and $\mathcal{F}: I\times X\to Y$ a continuous map with the following properties:
\begin{itemize}
\item [(1)] $\mathcal{F}(\alpha, 0)=0$ for all $\alpha\in I$;
\item [(2)] $\mathcal{F}_{\alpha}, \mathcal{F}_{f}$ and $\mathcal{F}_{\alpha f}$
exist and are continuous;
\item [(3)] $\mathcal{N}(\mathcal{F}_{f}(\alpha_c,0))$ and 
$Y\setminus \mathcal{R}(\mathcal{F}_{f}(\alpha_c,0))$
are one-dimensional, with the nullspace generated by $h^*$,and 
\item [(4)] $\mathcal{F}_{\alpha f}h^*\notin \mathcal{R}(\mathcal{F}_{f}(\alpha_c,0))$;
\end{itemize}
Then there exists a nontrivial continuously differentiable curve through $\left(\alpha_{c},0\right)$
\begin{align*}
    \left\{\left(\lambda\left(s\right),v\left(s\right)\right)|s\in\left(-\delta,\delta\right),~\left(\lambda\left(0\right),v\left(0\right)\right)=\left(\alpha_{c},0\right)\right\},
\end{align*}
such that $
\mathcal{F}\left(\lambda\left(s\right),v\left(s\right)\right)=0,~\text{for~}s\in\left(-\delta,\delta\right)$.
\end{theorem}

To verify that \((\alpha_c, H)\) is indeed a bifurcation point and to characterize the corresponding bifurcated solution, we begin by rewriting the nonlinear system \eqref{operator-equation} through the substitution  
\[
h(q, p) = H(p) + f(q, p).
\]  
We then define the following nonlinear operator \(\mathcal{F}(\alpha, f)\) in terms of \(\mathcal{G}\):
\begin{align*}
\mathcal{F}(\alpha, f) := \mathcal{G}(\alpha, H + f) = \left( \mathcal{F}_1(\alpha, f), \mathcal{F}_2(\alpha, f) \right),
\end{align*}
where the components are explicitly given by
\begin{align*}
\begin{aligned}
&\begin{aligned}
\mathcal{F}_1(\alpha,f)=
&f_{qq} (H'+f_p)^2
-(1+H+f)(H'+f_p)^2+(1+H+f)^2(H''+f_{pp})\\
&-2f_q(H'+f_p)f_{qp}+ 
(H''+f_{pp})(f_q)^2=0,\quad -1<p<0,
\end{aligned}\\
&\begin{aligned}
\mathcal{F}_2(\alpha,f)=&\left(1+H+f\right)^2\left(H'+f_p\right)^2\left[
2\lambda(\alpha,\gamma)+(1+H+f)^2-2\alpha \left(1+H+f\right)\right]
\\
&-p_{0}^2\left(f_q\right)^2-p_{0}^2\left(1+H+f\right)^2
=0, \quad p=0,
\end{aligned}
\end{aligned}
\end{align*}

Note that \(\mathcal{F}(\alpha, 0) = 0\). The Fréchet derivative of \(\mathcal{F}(\alpha, f)\) at \(f = 0\) is given by the pair  
\begin{align*}
\mathcal{F}_f(\alpha, 0) = \left( \mathcal{F}_{1f}(\alpha, 0), \mathcal{F}_{2f}(\alpha, 0) \right),
\end{align*}  
where the components \(\mathcal{F}_{1f}(\alpha, 0)\) and \(\mathcal{F}_{2f}(\alpha, 0)\) are defined as follows:
\begin{align}\label{congxinbian0820}
\begin{aligned}
&\begin{aligned}
\mathcal{F}_{1f}(\alpha,0)=& (1+H)^2\left(\gamma^2\partial_{qq}
+\partial_{pp}\right)
-2\gamma (1+H)^2\partial_p+\gamma^2(1+H)^2,
\end{aligned}\\ 
&\begin{aligned}
\mathcal{F}_{2f}(\alpha,0)=&2\gamma^2e^{3\gamma}\left(
2\lambda+e^{2\gamma}-2\alpha e^{\gamma}\right)
+\gamma^2e^{4\gamma}\left(
2e^{\gamma}-2\alpha\right)-2p_{0}^2e^{\gamma}\\
&+2\gamma e^{3\gamma}
\left(
2\lambda+e^{2\gamma}-2\alpha e^{\gamma}\right)\partial_p\Big|_{T}\\
=&\Big(2p_{0}^2\gamma^{-1}e^{\gamma}\partial_p-2\gamma^2e^{4\gamma}\left(\alpha
-e^{\gamma}\right)\Big)\Big|_{T}\\
=&2e^{\gamma}\Big(p_{0}^2\gamma^{-1}\partial_p-\beta(\alpha,\gamma) \Big)\Big|_{T}
\end{aligned}
\end{aligned}
\end{align}
where we have used the identities  
\[
2p_0^2 \gamma^{-1} e^{\gamma} + 2\gamma e^{3\gamma} \left(2\alpha e^{\gamma} - 2\lambda - e^{2\gamma}\right) = 0, \quad \beta(\alpha, \gamma) = \gamma^2 e^{3\gamma} \left( \alpha - e^{\gamma} \right).
\]  
The condition \(\sigma_1(\alpha_c) = 0\) implies that the null space of \(\mathcal{F}_f(\alpha, 0)\) at \(\alpha = \alpha_c\) is nontrivial.

Subsequently, we establish two lemmas to prove that the null space of \(\mathcal{F}_f(\alpha_c, 0)\) is one-dimensional and that the codimension of the range \(\mathcal{R}(\mathcal{F}_f(\alpha_c, 0))\) in \(Y\) is also one.

\begin{lemma}\label{yige0620}
For \(\alpha = \alpha_c\), the null space of \(\mathcal{F}_f(\alpha, 0)\) is one-dimensional, and a representative element can be chosen as  
\begin{align}\label{hstar}
h^* = \left( e^{2\gamma p} - e^{-2\gamma} \right) \cos q.
\end{align}
\end{lemma}

\begin{proof}
The existence of a nontrivial element in the null space follows from the variational formulation \eqref{eigen-variation}, which guarantees at least one function of the form  
\[
f(q, p) = M(p) \cos(kq),
\]  
where \(M(p)\) is determined by \eqref{eigen-variation}. It remains to prove uniqueness. Let \(\phi(q, p) \in \mathcal{N}(\mathcal{F}_f(\alpha_c, 0))\). Expanding \(\phi\) in a cosine series yields  
\[
\phi(q, p) = \sum_{k=0}^{\infty} \phi_k(p) \cos(kq).
\]  
Substituting into the linearized operator and applying the boundary conditions, we obtain  
\[
\sum_{k=0}^{\infty} \left[ (H+1)^2 \left( \gamma^2 \partial_{qq} + \partial_{pp} \right) - 2\gamma (H+1)^2 \partial_p + \gamma^2 (H+1)^2 \right] \phi_k(p) \cos(kq) = 0,
\]  
and  
\[
\sum_{k=0}^{\infty} 2e^{\gamma} \left( p_0^2 \gamma^{-1} \partial_p - \beta(\alpha_c, \gamma) \right) \phi_k(p) \cos(kq) = 0, \quad p = 0,
\]  
along with the periodic and vanishing boundary conditions at \(p = -1\). This implies that each mode \(\phi_k(p)\) satisfies  
\begin{align}\label{xuyao0620}
\begin{cases}
\left[ (H+1)^{-2} \partial_{pp} - 2\gamma (H+1)^{-2} \partial_p + \gamma^2 (1 - k^2) (H+1)^{-2} \right] \phi_k(p) = 0, \\
\phi_k(-1) = 0, \\
\beta(\alpha_c, \gamma) \phi_k(0) - p_0^2 \gamma^{-1} \phi_k'(0) = 0.
\end{cases}
\end{align}

If for some \(k > 1\) such that \(\phi_k \neq 0\), we obtain  
\[
\begin{aligned}
0 = {} & -\beta(\alpha_c, \gamma) \gamma e^{-2\gamma} (\phi_k(0))^2 / p_0^2 + \int_{-1}^0 (H + 1)^{-2} (\phi_k'(p))^2  dp \\
& + \gamma^2 (k^2 - 1) \int_{-1}^0 (H + 1)^{-2} (\phi_k(p))^2  dp.
\end{aligned}
\]  
Since \(k > 1\), the term \(\gamma^2 (k^2 - 1) \int_{-1}^0 (H + 1)^{-2} (\phi_k(p))^2  dp\) is positive, implying  
\[
0 > -\beta(\alpha_c, \gamma) \gamma e^{-2\gamma} (\phi_k(0))^2 / p_0^2 + \int_{-1}^0 (H + 1)^{-2} (\phi_k'(p))^2  dp.
\]  
This leads to the contradiction  
\begin{align*}
0 = -\tilde{\sigma}_1(\alpha_c) := \inf_{M \in \mathcal{H}} \frac{ 
  -\beta(\alpha_c, \gamma) \gamma e^{-2\gamma} M^2(0) / p_0^2 
  + \int_{-1}^0 (H + 1)^{-2} M_p^2  dp 
}{ \int_{-1}^0 (H + 1)^{-4} M^2  dp } < 0,
\end{align*}  
as the infimum is strictly negative.

For \(k = 0\), the system reduces to  
\begin{align*}
\begin{cases}
\left( \partial_{pp} - 2\gamma \partial_p + \gamma^2 \right) \phi_0(p) = 0, \\
\beta(\alpha_c, \gamma) \phi_0(0) - p_0^2 \gamma^{-1} \phi_0'(0) = 0, \\
\phi_0(-1) = 0,
\end{cases}
\end{align*}  
which has the general solution  
\[
\phi_0(p) = A_1 e^{\gamma p} + A_2 p e^{\gamma p}.
\]  
Applying the boundary conditions at \(p = -1\) and \(p = 0\), we obtain  
\begin{align*}
\begin{cases}
A_1 e^{-\gamma} - A_2 e^{-\gamma} = 0, \\
\beta(\alpha_c, \gamma)  A_1  - p_0^2 \gamma^{-1} \left( \gamma A_1 + A_2 \right) = 0.
\end{cases}
\end{align*}  
Simplifying and using the expression for \(\alpha_c\) from \eqref{critical-alpha} with the fact $(1-\gamma)e^{2\gamma}\neq (1+\gamma)$ as $0<\gamma<1$, 
we find  
\[
A_1 = A_2 = 0 \quad \Rightarrow \quad \phi_0(p) = 0.
\]  

Finally, to determine the element of \(\mathcal{N}(\mathcal{F}_f(\alpha_c, 0))\), we consider \(k = 1\) in \eqref{xuyao0620}. This yields  
\begin{align*}
\phi_1''(p) - 2\gamma \phi_1'(p) = 0, \quad \phi_1(-1) = 0, \quad \beta(\alpha_c, \gamma) \phi_1(0) - \frac{p_0^2}{\gamma} \phi_1'(0) = 0.
\end{align*}  
The general solution is \(\phi_1(p) = C_1 e^{2\gamma p} + C_2\). Applying the boundary conditions at \(p = -1\) and  \(p = 0\) gives
\(C_1 e^{-2\gamma} + C_2 = 0\) and
\(\beta(\alpha_c, \gamma) (C_1 + C_2) - 2p_0^2 C_1 = 0\).  
Using the expression \(\alpha_c = e^{\gamma} + \frac{2p_0^2}{\gamma^2 e^{\gamma} (e^{2\gamma} - 1)}\) and \(\beta(\alpha_c, \gamma) = \gamma^2 e^{3\gamma} (\alpha_c - e^{\gamma})\), one can verify that $
C_1 = 1$ and $C_2 = -e^{-2\gamma}$
satisfy both conditions. Thus, the representative element is $
\phi_1(p) = e^{2\gamma p} - e^{-2\gamma}$.
\end{proof}

\begin{lemma}\label{range0507}
The pair \((u, b)\) belongs to the range of the linear operator \(\mathcal{F}_f(\alpha_c, 0)\) if and only if it satisfies the following orthogonality condition:
\begin{align}\label{orthongonality0507}
\int_{\Omega} u (1 + H)^{-4} h^{*}  dq  dp - \frac{1}{2} \int_{T} (1 + H)^{-3} \frac{\gamma h^{*}}{p_0^2} b  dq = 0,
\end{align}
where \(h^{*}\), given by \eqref{hstar}, is the generator of the null space of \(\mathcal{F}_f(\alpha_c, 0)\).
\end{lemma}
\begin{proof}
For convenience, we firstly rewrite the linear operator \(\mathcal{F}_{f}\left(\alpha_{c},0\right)=\left(\mathcal{F}_{1f}\left(\alpha_{c},0\right),\mathcal{F}_{2f}\left(\alpha_{c},0\right)\right)\) in the following form 
\begin{align}\label{self0507}
\begin{aligned}
&\mathcal{F}_{1f}\left(\alpha_{c},0\right)h=\left(1+H\right)^3\left[\left(1+H\right)^{-1}h\right]_{pp}+\gamma^2\left(1+H\right)^2 h_{qq},
\\
&\mathcal{F}_{2f}\left(\alpha_{c},0\right)h=2e^{\gamma}\left[\frac{p_{0}^2}{\gamma}h_{p}-\beta\left(\alpha_{c},\gamma\right)h\right]\bigg{|}_{T}, ~\text{for~}h\in X.
\end{aligned}
\end{align}

"\(\Longrightarrow\)" Assume that the pair \((u, b)\) belongs to the range of 
\(\mathcal{F}_f(\alpha_c, 0)\). Then there exists a function \(v\) vanishing on \(B\) and periodic in \(q\) such that  
\begin{align*}
\begin{aligned}
&(1 + H)^3 \left[ (1 + H)^{-1} v \right]_{pp} + \gamma^2 (1 + H)^2 v_{qq} = u, \quad \text{in } \Omega, \\
&2e^{\gamma} \left[ \frac{p_0^2}{\gamma} v_p - \beta(\alpha_c, \gamma) v \right] \bigg|_{T} = b, \quad \text{on } T.
\end{aligned}
\end{align*}  
A direct computation using integration by parts and the boundary conditions yields: 
\begin{align*}
\begin{aligned}
\int_{\Omega}u\left(1+H\right)^{-4}h^{*}dqdp&=
\int_{\Omega}\left(1+H\right)^{-1}v\left\{\left[\left(1+H\right)^{-1}h^{*}\right]_{pp}+\gamma^2\left(1+H\right)^{-1}h_{qq}^{*}\right\}dqdp
\\
&\quad+\int_{T}\left[\left(1+H\right)^{-1}v\right]_{p}\left(1+H\right)^{-1}h^{*}dq\\&\quad-\int_{T}
\left(1+H\right)^{-1}v\left[\left(1+H\right)^{-1}h^{*}\right]_{p}dq
\\
&=\int_{T}\left[\left(1+H\right)^{-1}v\right]_{p}\left(1+H\right)^{-1}h^{*}dq\\&\quad-\int_{T}
\left(1+H\right)^{-1}v\left[\left(1+H\right)^{-1}h^{*}\right]_{p}dq
\end{aligned}
\end{align*}
where we have used the fact that \(h^{*}\) spans the one-dimensional null space of \(\mathcal{F}_{f}(\alpha_{c}, 0)\). Furthermore, on the boundary \(T\), the following identities hold:
\begin{align*}
\begin{aligned}
&p_0^2 h^{*}_p = \beta(\alpha_c, \gamma) \gamma h^{*}, \\
&\left[ (1 + H)^{-1} h^{*} \right]_p = (1 + H)^{-1} \gamma h^{*} \left( \frac{\beta(\alpha_c, \gamma)}{p_0^2} - 1 \right), \\
&\left[ (1 + H)^{-1} v \right]_p = (1 + H)^{-1} v_p - \gamma (1 + H)^{-1} v.
\end{aligned}
\end{align*}
Consequently, we obtain
\begin{align*}
\begin{aligned}
\int_{\Omega}u\left(1+H\right)^{-4}h^{*}dqdp-\frac{1}{2}\int_{T}\left(1+H\right)^{-3}\frac{\gamma h^{*}}{p_{0}^2}bdq=0.
\end{aligned}
\end{align*}

"\(\Longleftarrow\)"Suppose that the orthogonality condition \eqref{orthongonality0507} is satisfied. We consider the boundary value problem  
\begin{align}\label{given0507}
\begin{aligned}
&(1 + H)^3 \left[ (1 + H)^{-1} v \right]_{pp} + \gamma^2 (1 + H)^2 v_{qq} = u, \quad \text{in } \Omega, \\
&2e^{\gamma} \left[ \frac{p_0^2}{\gamma} v_p - \beta(\alpha_c, \gamma) v \right] \bigg|_{T} = b, \quad \text{on } T,
\end{aligned}
\end{align}  
where \((u, b) \in Y\) is given. To leverage elliptic theory, we introduce the transformation  
\[
\widetilde{v} := (1 + H)^{-1} v, \quad \text{that is,} \quad v = (1 + H) \widetilde{v}.
\]  
Substituting this into \eqref{given0507}, we obtain the equivalent system  
\begin{align*}
\begin{aligned}
&(1 + H)^3 \widetilde{v}_{pp} + \gamma^2 (1 + H)^3 \widetilde{v}_{qq} = u, \quad \text{in } \Omega, \\
&2e^{2\gamma} \left[ \frac{p_0^2}{\gamma} \widetilde{v}_p - \left( p_0^2 - \beta(\alpha_c, \gamma) \right) \widetilde{v} \right] \bigg|_{T} = b, \quad \text{on } T.
\end{aligned}
\end{align*}

Furthermore, for any given \(\varepsilon > 0\), we introduce the following regularized auxiliary problem:
\begin{align}\label{fuzhu0507}
\begin{aligned}
&(1 + H)^3 \widetilde{v}^{(\varepsilon)}_{pp} + \gamma^2 (1 + H)^3 \widetilde{v}^{(\varepsilon)}_{qq} - \varepsilon \widetilde{v}^{(\varepsilon)} = u, \quad \text{in } \Omega, \\
&2e^{2\gamma} \left[ \frac{p_0^2}{\gamma} \widetilde{v}^{(\varepsilon)}_p - \left( p_0^2 - \beta(\alpha_c, \gamma) \right) \widetilde{v}^{(\varepsilon)} \right] \bigg|_{T} = b, \quad \text{on } T.
\end{aligned}
\end{align}
By standard elliptic theory \cite{Gilbarg2001}, the regularized problem \eqref{fuzhu0507} admits a unique solution \(\widetilde{v}^{(\varepsilon)}\). Equivalently, the following equation has a unique solution \(v^{(\varepsilon)}\):
\begin{align}\label{fuzhu10507}
\begin{aligned}
&\begin{aligned}
&\left(1+H\right)^3\left[\left(1+H\right)^{-1}v^{\left(\varepsilon\right)}\right]_{pp}+\gamma^2\left(1+H\right)^2\left[\left(1+H\right)^{-1}v^{\left(\varepsilon\right)}\right]_{qq}\\&-\varepsilon\left[\left(1+H\right)^{-1}v^{\left(\varepsilon\right)}\right]=u,~\text{in}~\Omega,
\end{aligned}
\\
&\begin{aligned}
&2e^{\gamma}\left[\frac{p_{0}^2}{\gamma}v^{\left(\varepsilon\right)}_{p}-\beta\left(\alpha_{c},\gamma\right)v^{\left(\varepsilon\right)}\right]\bigg{|}_{T}=b,~~~\text{on~}~\text{T}.
\end{aligned}
\end{aligned}
\end{align}
We claim that the family \(\left\{ v^{(\varepsilon)} \right\}\) is bounded in \(C_{\text{per}}^{1+s}(\overline{\Omega})\). Suppose, for contradiction, that  
\[
\left\| v^{(\varepsilon)} \right\|_{C_{\text{per}}^{1+s}(\overline{\Omega})} \to +\infty \quad \text{as} \quad \varepsilon \to 0^+.
\]  
Define the normalized function  
\[
u^{(\varepsilon)} = \frac{v^{(\varepsilon)}}{\left\| v^{(\varepsilon)} \right\|_{C_{\text{per}}^{1+s}(\overline{\Omega})}}.
\]  
Then, dividing \eqref{fuzhu10507} by \(\left\| v^{(\varepsilon)} \right\|_{C_{\text{per}}^{1+s}(\overline{\Omega})}\) and taking the limit as \(\varepsilon \to 0^+\), we obtain
\begin{align*}
\begin{aligned}
&\left(1+H\right)^3\left[\left(1+H\right)^{-1}u^{\left(\varepsilon\right)}\right]_{pp}+\gamma^2\left(1+H\right)^2\left[\left(1+H\right)^{-1}u^{\left(\varepsilon\right)}\right]_{qq}\rightarrow 0,~\text{in}~C_{\text{per}}^{1+s}\left(\overline{\Omega}\right),
\\
&2e^{\gamma}\left[\frac{p_{0}^2}{\gamma}u^{\left(\varepsilon\right)}_{p}-\beta\left(\alpha_{c},\gamma\right)u^{\left(\varepsilon\right)}\right]\bigg{|}_{T}\rightarrow 0,~~\text{in~}C_{\text{per}}^{1+s}\left(T\right).
\end{aligned}
\end{align*}
which, combined with the Schauder estimates \cite{Gilbarg2001}, implies that the family \(\left\{ u^{(\varepsilon)} \right\}\) is bounded in \(C_{\text{per}}^{2+s}(\overline{\Omega})\). By the compact embedding of Hölder spaces, we may extract a subsequence of \(\left\{ u^{(\varepsilon)} \right\}\) (still denoted by \(\left\{ u^{(\varepsilon)} \right\}\)) that converges to some function \(u \in C_{\text{per}}^{2}(\overline{\Omega})\). Consequently, \(u\) satisfies the limiting equation:
\begin{align*}
\begin{aligned}
&\left(1+H\right)^3\left[\left(1+H\right)^{-1}u\right]_{pp}+\gamma^2\left(1+H\right)^2\left[\left(1+H\right)^{-1}u\right]_{qq}= 0,~\text{in}~C_{\text{per}}^{1+s}\left(\overline{\Omega}\right),
\\
&2e^{\gamma}\left[\frac{p_{0}^2}{\gamma}u_{p}-\beta\left(\alpha_{c},\gamma\right)u\right]\bigg{|}_{T}=0,~~\text{in~}C_{\text{per}}^{1+s}\left(T\right),
\end{aligned}
\end{align*}
which implies that \(u\) belongs to the null space of \(\mathcal{F}_f(\alpha_c, 0)\). Therefore, \(u\) must be a scalar multiple of \(h^*\). However, if we multiply \eqref{fuzhu10507} by $(1 + H)^{-4} h^*$ and integrate over $\Omega$, we obtain
\[
\int_{\Omega}\left(1+H\right)^{-5}v^{\left(\varepsilon\right)}h^{*}dqdp=0,
\]
where the orthogonality condition \eqref{orthongonality0507} has been applied. It follows that  
\[
\int_{\Omega} (1 + H)^{-5} u h^{*}  dq  dp = 0.
\]  
Since \(u\) is a scalar multiple of \(h^{*}\), this implies \(u = 0\). However, this contradicts the fact that  
\[
\| u \|_{C_{\text{per}}^{1+s}(\overline{\Omega})} = 1.
\]  
This contradiction establishes the claim that the family \(\left\{ v^{(\varepsilon)} \right\}\) is bounded in \(C_{\text{per}}^{1+s}(\overline{\Omega})\).

This boundedness ensures the existence of a strongly convergent subsequence of \(\left\{ v^{(\varepsilon)} \right\}\) that converges to a limit \(v\) in \(C_{\text{per}}^{1}(\overline{\Omega})\). Passing to the limit as \(\varepsilon \to 0^+\) in \eqref{fuzhu10507}, we obtain  
\begin{align*}
\begin{aligned}
&(1 + H)^3 \left[ (1 + H)^{-1} v \right]_{pp} + \gamma^2 (1 + H)^2 \left[ (1 + H)^{-1} v \right]_{qq} = u, \quad \text{in } \Omega, \\
&2e^{\gamma} \left[ \frac{p_0^2}{\gamma} v_p - \beta(\alpha_c, \gamma) v \right] \bigg|_{T} = b, \quad \text{on } T,
\end{aligned}
\end{align*}  
in the sense of distributions. By elliptic regularity theory \cite{Gilbarg2001}, it follows that \(v \in X\). This demonstrates that the pair $(u, b)$ belongs to the range of $\mathcal{F}_f(\alpha_c, 0)$.

\end{proof}

\begin{remark}\label{erge0620}
\autoref{range0507} implies that the codimension of \(\mathcal{R}(\mathcal{F}_f(\alpha_c, 0))\) in \(Y\) is one. For any two elements \((u_1, b_1), (u_2, b_2) \in Y\), we define the inner product  
\[
\langle (u_1, b_1), (u_2, b_2) \rangle := \int_{\Omega} u_1 u_2  dq  dp + \int_{T} b_1 b_2  dq.
\]  
Then, by \autoref{range0507}, a representative element of \(Y \setminus \mathcal{R}(\mathcal{F}_f(\alpha_c, 0))\) can be chosen as  
\[
\mathbf{h}_0 = \left( (1 + H)^{-4} h^*, -\frac{(1 + H)^{-3} \gamma h^*}{2p_0^2} \right) / C_0,
\]  
where the normalization constant \(C_0\) is given by  
\[
C_0 = \sqrt{\pi \left[ \frac{e^{-12 \gamma } \left(e^{2 \gamma }-1\right)^3 \left(e^{2 \gamma }+3\right)}{24 \gamma }+ \frac{\gamma^2 (e^{-3\gamma} - e^{-5\gamma})^2}{4p_0^4} \right]},
\]  
chosen to ensure \(\langle \mathbf{h}_0, \mathbf{h}_0 \rangle = 1\).
\end{remark}
\subsection{Proof of \autoref{dingli08201} }\label{433}

\begin{proof}
We now apply \autoref{the21} to establish our local bifurcation theorem, \autoref{dingli08201}. Through standard computations derived from \eqref{congxinbian0820}, we obtain  
\begin{align*}
\mathcal{F}_{f\alpha}(\alpha_c, 0) = \left( 0, -2\gamma^2 e^{4\gamma} \big|_{T} \right).
\end{align*}  
Hence, to complete the proof of parts (1) and (2) of \autoref{dingli08201}, it suffices to verify that  
\[
\mathcal{F}_{f\alpha}(\alpha_c, 0) h^*(q, p) \notin \mathcal{R}\left( \mathcal{F}_f(\alpha_c, 0) \right),
\]  
where \(h^*(q, p)\) is the generator of the null space as defined in \eqref{hstar}.

Suppose, for contradiction, that  
\[
\mathcal{F}_{f\alpha}(\alpha_c, 0) h^*(q, p) = \mathcal{F}_f(\alpha_c, 0) h
\]  
for some \(h \in X\). It then follows that  
\[
\mathcal{F}_f(\alpha_c, 0) h = \left( 0, -2\gamma^2 e^{4\gamma} B (e^{2\gamma} - 1) \cos q \right),
\]  
which implies the following system:  
\[
\begin{cases}
\left[ (1 + H)^2 \left( \gamma^2 \partial_{qq} + \partial_{pp} \right) - 2\gamma (1 + H)^2 \partial_p + \gamma^2 (1 + H)^2 \right] h = 0 \quad \text{in } \Omega,\\
2e^{\gamma} \left( p_0^2 \gamma^{-1} \partial_p - \beta(\alpha_c, \gamma) \right) h \Big|_{T} = -2\gamma^2 e^{4\gamma} B (e^{2\gamma} - 1) \cos q \quad \text{on } T.
\end{cases}
\]

Expanding \(h(q, p)\) in a Fourier cosine series, we express  
\[
h(q, p) = \sum_{k=0}^{\infty} H_k(p) \cos(kq).
\]  
Substituting this expansion into the governing equations, we derive  
\[
\begin{cases}
\sum_{k=0}^{\infty} \left[ (H + 1)^2 \left( \gamma^2 \partial_{qq} + \partial_{pp} \right) - 2\gamma (H + 1)^2 \partial_p + \gamma^2 (H + 1)^2 \right] H_k(p) \cos(kq) = 0,\\
\sum_{k=0}^{\infty} 2e^{\gamma} \left[ p_0^2 \gamma^{-1} \partial_p - \beta(\alpha_c, \gamma) \right] H_k(p) \cos(kq) = -2\gamma^2 e^{4\gamma} B (e^{2\gamma} - 1) \cos q, \quad p = 0,
\end{cases}
\]
subject to periodic boundary conditions in \(q\) and the Dirichlet condition \(h = 0\) at \(p = -1\). This decomposition implies that for \(k \neq 1\),  
\begin{align}\label{jieyouci0915}
\begin{cases}
\left[ (H + 1)^{-2} \partial_{pp} - 2\gamma (H + 1)^{-2} \partial_p + \gamma^2 (1 - k^2) (H + 1)^{-2} \right] H_k(p) = 0, \\
\beta(\alpha_c, \gamma) H_k(0) - p_0^2 \gamma^{-1} H_k'(0) = 0, \\
H_k(-1) = 0,
\end{cases}
\end{align}  
which has a general solution 
\begin{align*}
H_{k}(p)=C_{k1}e^{\gamma(1-k)p}+C_{k2}e^{\gamma(k+1)p}.
\end{align*}
Then from $\eqref{jieyouci0915}_{2}$ and $\eqref{jieyouci0915}_{3}$, we have 
\begin{align*}
\begin{aligned}
&C_{k1}e^{\gamma(k-1)}+C_{k2}e^{-\gamma(1+k)}=0,
\\
&\left(\frac{2e^{2\gamma}}{e^{2\gamma}-1}+(k-1)\right)C_{k1}+\left(\frac{2e^{2\gamma}}{e^{2\gamma}-1}-(1+k)\right)C_{k2}=0,
\end{aligned}
\end{align*}
which has a trivial solution if and only if 
\begin{align*}
    \begin{vmatrix}
        e^{\gamma(k-1)}  &e^{-\gamma(k+1)}
        \\
        \frac{2e^{2\gamma}}{e^{2\gamma}-1}+(k-1) &\frac{2e^{2\gamma}}{e^{2\gamma}-1}-(1+k)
    \end{vmatrix}\neq 0.
\end{align*}
That is,
\begin{align*}
    \left(e^{2k\gamma}-e^{2\gamma}\right)(1+k)\neq \left(e^{2(1+k)\gamma}-1\right)(k-1),
\end{align*}
which can be verified by taking several derivatives of  $\xi(r)= \left(e^{2k\gamma}-e^{2\gamma}\right)(1+k)-\left(e^{2(1+k)\gamma}-1\right)(k-1)$ with respect to $r$.

for \(k = 1\),  
\begin{align}\label{22}
\begin{cases}
\left[ (H + 1)^{-2} \partial_{pp} - 2\gamma (H + 1)^{-2} \partial_p \right] H_1(p) = 0, \quad H_1(-1) = 0, \\
\beta(\alpha_c, \gamma) H_1(0) - p_0^2 \gamma^{-1} H_1'(0) = -\gamma^2 e^{3\gamma} B (e^{2\gamma} - 1),
\end{cases}
\end{align}  
and for \(k = 0\),  
\begin{align*}
\begin{cases}
\left( \partial_{pp} - 2\gamma \partial_p + \gamma^2 \right) H_0(p) = 0, \quad H_0(-1) = 0, \\
\beta(\alpha_c, \gamma) H_0(0) - p_0^2 \gamma^{-1} H_0'(0) = 0,
\end{cases}
\end{align*}
which has a trivial solution as $0<\gamma<1$.

For all \(k \neq 1\), the only solution satisfying the required boundary conditions is the trivial solution:  
\[
H_k(p) = 0.
\]

Note that the system \eqref{22} is equivalent to  
\begin{align}\label{new-equation-per-ode-222}
\begin{cases}
\partial_p \left[ (H + 1)^{-2} \partial_p H_1 \right] = 0, \\
\beta(\alpha_c, \gamma) H_1(0) - p_0^2 \gamma^{-1} H_1'(0) = -\gamma^2 e^{3\gamma} B (e^{2\gamma} - 1), \\
H_1(-1) = 0.
\end{cases}
\end{align}  
The general solution of the differential equation \(\partial_p \left[ (1 + H)^{-2} \partial_p H_1 \right] = 0\) is  
\[
H_1(p) = B C_1 e^{2\gamma(p + 1)} + B C_2,
\]  
where \(C_1\) and \(C_2\) are constants. From this, it follows that the boundary value problem \eqref{new-equation-per-ode-222} admits a nontrivial solution if and only if the determinant condition  
\[
\begin{vmatrix}
1 & 1 \\
e^{2\gamma} \beta(\alpha_c, \gamma) - 2p_0^2 e^{2\gamma} & \beta(\alpha_c, \gamma)
\end{vmatrix} \neq 0
\]  
is satisfied. However, this contradicts the critical condition \eqref{critical-alpha-1}, which requires precisely that this determinant vanishes. Therefore, we conclude that  
\[
\mathcal{F}_{f\alpha}(\alpha_c, 0) h^*(q, p) \notin \mathcal{R}\left( \mathcal{F}_f(\alpha_c, 0) \right).
\]

We now proceed to prove part (3) of \autoref{dingli08201}, which concerns the type of bifurcation. According to Corollary I.5.2 in \cite{Kielhoefer2012}, the tangent vector to the bifurcation curve at the critical point is given by \((\dot{\lambda}(0), \dot{x}(0)) = (\mathcal{X}, h^*)\), where \(h^*\) generates the null space of \(\mathcal{F}_f(\alpha_c, 0)\). If \(\mathcal{X} = 0\), the bifurcation is not transcritical.  

To analyze the bifurcation type, we decompose the spaces \(X\) and \(Y\) as follows:
\begin{align}\label{fenjie0620}
X = \mathcal{N}\left( \mathcal{F}_f(\alpha_c, 0) \right) \oplus X_0, \quad Y = \mathcal{R}\left( \mathcal{F}_f(\alpha_c, 0) \right) \oplus Y_0.
\end{align}
From \autoref{yige0620} and  \autoref{erge0620}, we have  
\[
\mathcal{N}\left( \mathcal{F}_f(\alpha_c, 0) \right) = \text{span} \left\{ \hat{h}^* \right\}, \quad Y_0 = \text{span} \left\{ \mathbf{h}_0 \right\}, \quad \text{dim} \, \mathcal{N}\left( \mathcal{F}_f(\alpha_c, 0) \right) = \text{dim} \, Y_0 = 1,
\]  
where \(\hat{h}^* = h^* / \hat{C}\) with normalization constant  
\[
\hat{C} = \sqrt{ \pi \left[ \frac{1}{4\gamma} + \frac{3 + 4\gamma}{4\gamma e^{4\gamma}} - \frac{1}{\gamma e^{2\gamma}} \right] },
\]  
chosen such that \(\int_{\Omega} (\hat{h}^*)^2  dq  dp = 1\).

These decompositions, in turn, define projections
\begin{align}\label{suanzi0624}
    \begin{aligned}
    &\mathbf{P}:~X\rightarrow \mathcal{N}\left(\mathcal{F}_{f}\left(\alpha_{c},0\right)\right)~\text{along}~X_{0},
    \\
    &\mathbf{Q}:~Y\rightarrow Y_{0},~\text{along}~\mathcal{R}\left(\mathcal{F}_{f}\left(\alpha_{c},0\right)\right).
    \end{aligned}
\end{align}

Employing the formula I.6.3 in \cite{Kielhoefer2012}, we have
\begin{align}\label{kongzhi0820}
\dot{\lambda}\left(0\right)=\mathcal{X}=-\frac{1}{2}
\frac{\langle\mathcal{F}_{ff}\left(\alpha_{c},0\right)\left[\hat{h}^{*},\hat{h}^{*}\right],\mathbf{h}_{0} \rangle}{\langle \mathcal{F}_{\alpha f}\left(\alpha_{c},0\right)\hat{h}^{*}, \mathbf{h}_{0} \rangle}.
\end{align}
Through a series of computation, we obtain 
\begin{align*}
\begin{aligned}
&\begin{aligned}
\mathcal{F}_{1ff}\left(\alpha_{c},0\right)\left[\hat{h}^{*},\hat{h}^{*}\right]=\frac{2\gamma^2 e^{(p-3)\gamma}\left(1-2e^{2\left(p+1\right)\gamma} \cos{2q}-3e^{4\left(p+1\right)\gamma}\right)}{\hat{C}^2},
\end{aligned}
\\
&\begin{aligned}
\mathcal{F}_{2ff}\left(\alpha_{c},0\right)\left[\hat{h}^{*},\hat{h}^{*}\right]&=\frac{z_{1}\left(\gamma\right)+z_{2}\left(\gamma\right)\cos{2 q}}{\hat{C}^2 },
\end{aligned}
\\
&\begin{aligned}
   & \mathcal{F}_{\alpha f}\left(\alpha_{c},0\right)\hat{h}^{*}=\left(0,-2\gamma^2 e^{2\gamma}\left(e^{2\gamma}-1\right)\cos{q}/\hat{C}\right),\\&
    \langle\mathcal{F}_{\alpha f}\left(\alpha_{c},0\right)\hat{h}^{*}, \mathbf{h}_{0}\rangle=\frac{\gamma^3 \pi \left(e^{2\gamma}-1\right)^2}{\hat{C} C_{0} e^{3\gamma} p_{0}^2 },
\end{aligned}
\end{aligned}
\end{align*}
where $z_{j}\left(\gamma\right)$ with $j=1,2$ are given by
\begin{align*}
&z_{1}\left(\gamma\right)=\gamma ^2-2 \gamma ^2 e^{2 \gamma }+\gamma ^2 e^{4 \gamma }-p_{0}^2 e^{-4 \gamma }+2 p_{0}^2 e^{-2 \gamma }-13 p_{0}^2,
\\
&z_{2}\left(\gamma\right)=\gamma ^2-2 \gamma ^2 e^{2 \gamma }+\gamma ^2 e^{4 \gamma }+p_{0}^2 e^{-4 \gamma }-2 p_{0}^2 e^{-2 \gamma }-11 p_{0}^2.
\end{align*}
Thus, we can obtain $\mathcal{X}=0$. That is, the bifurcation is not transcritical. 

Since \(\mathcal{X}=0\), it follows from formula I.6.11 in  \cite{Kielhoefer2012} that 
the bifurcation type is determined by computing a parameter \(\mathcal{O} \), defined as:\begin{align*}
\mathcal{O} = \mathcal{O}_1 + \mathcal{O}_2,
\end{align*}  
where \(\mathcal{O}_1\) and \(\mathcal{O}_2\) are given by  
\begin{align*}
\begin{aligned}
&\mathcal{O}_1 = -\frac{1}{3} \frac{ \langle \mathcal{F}_{fff}(\alpha_c, 0)[\hat{h}^*, \hat{h}^*, \hat{h}^*], \mathbf{h}_0 \rangle }{ \langle \mathcal{F}_{\alpha f}(\alpha_c, 0) \hat{h}^*, \mathbf{h}_0 \rangle }, \\
&\mathcal{O}_2 = \frac{ \langle \mathcal{F}_{ff}(\alpha_c, 0) \left[ \hat{h}^*, (\mathbf{I} - \mathbf{P}) \left( \mathcal{F}_f(\alpha_c, 0) \right)^{-1} (\mathbf{I} - \mathbf{Q}) \mathcal{F}_{ff}(\alpha_c, 0)[\hat{h}^*, \hat{h}^*] \right], \mathbf{h}_0 \rangle }{ \langle \mathcal{F}_{\alpha f}(\alpha_c, 0) \hat{h}^*, \mathbf{h}_0 \rangle }.
\end{aligned}
\end{align*}  
According to \cite{Kielhoefer2012}, if \(\mathcal{O} > 0\), the bifurcation is supercritical and subcritical if \(\mathcal{O} < 0\). In both cases, the bifurcation diagram exhibits a pitchfork structure.

In order to obtain $\mathcal{O}$, we need to compute \(\mathcal{O}_1 \) and \(\mathcal{O}_1 \). For \(\mathcal{O}_1 \), a direct computation gives that
\[
\begin{aligned}
\mathcal{O}_{1}=&\frac{1}{\hat{C}^2 \left(e^{2 \gamma }-1\right)^2}\bigg{(}
\frac{9 e^{3 \gamma }}{2}-15 e^{\gamma }-\frac{4 p_{0}^2 e^{-7 \gamma }}{\gamma ^2}+e^{-5 \gamma } \left(\frac{3}{2}-\frac{12 p_{0}^2}{\gamma ^2}-\frac{8 p_{0}^2}{\gamma }\right)
\\
&+e^{-3 \gamma } \left(\frac{44 p_{0}^2}{\gamma ^2}-9\right)+e^{-\gamma } \left(18-\frac{28 p_{0}^2}{\gamma ^2}\right)
\bigg{)}.
\end{aligned}
\]

For the term \(\mathcal{O}_2 \), it is necessary to compute the expression
\[
(\mathbf{I}-\mathbf{P})\left(\mathcal{F}_{f}\left(\alpha_{c},0\right)\right)^{-1}\left(\mathbf{I}-\mathbf{Q}\right)\mathcal{F}_{ff}\left(\alpha_{c},0\right)\left[\hat{h}^{*},\hat{h}^{*}\right].
\]
Since \(\mathcal{F}_{ff}\left(\alpha_{c},0\right)\left[\hat{h}^{*},\hat{h}^{*}\right]\in \mathcal{R}\left(\mathcal{F}_{f}\left(\alpha_{c},0\right)\right)\), we first conclude that 
\[
\left(\mathbf{I}-\mathbf{Q}\right)\mathcal{F}_{ff}\left(\alpha_{c},0\right)\left[\hat{h}^{*},\hat{h}^{*}\right]=\mathcal{F}_{ff}\left(\alpha_{c},0\right)\left[\hat{h}^{*},\hat{h}^{*}\right].
\]
Hence, the computation of  \(\mathcal{O}_2\)
reduces to finding a function $h$ that satisfies the equation
\[
\mathcal{F}_{f}\left(\alpha_{c},0\right)h=\mathcal{F}_{ff}\left(\alpha_{c},0\right)\left[\hat{h}^{*},\hat{h}^{*}\right].
\]
 That is, we will solve the following partial differential equations,
\begin{align}\label{buweiyi0713}
\begin{cases}
\begin{aligned}
&\left(1+H\right)^3\left[(1+H)^{-1}h\right]_{pp}+\gamma^2\left(1+H\right)^2 h_{qq}\\&=\frac{2\gamma^2 e^{(p-3)\gamma}\left(1-2e^{2\left(p+1\right)\gamma} \cos{2q}-3e^{4\left(p+1\right)\gamma}\right)}{\hat{C}^2},
\end{aligned}
\\
\begin{aligned}
2e^{\gamma}\left[\frac{p_{0}^2}{\gamma}h_{p}-\beta\left(\alpha_{c},\gamma\right)h\right]\Big|_{p=0}
&=\frac{z_{1}\left(\gamma\right)+z_{2}\left(\gamma\right)\cos{2 q}}{\hat{C}^2 }.
\end{aligned}
\end{cases}
\end{align}

We observe that the solution to \eqref{buweiyi0713} is not unique; however, this non-uniqueness does not affect the expression of \(\mathcal{O}\). Indeed, if \(h_1\) and \(h_2\) are two solutions of \eqref{buweiyi0713}, then \(h_1 - h_2 \in \mathcal{N}(\mathcal{F}_f(\alpha_c, 0))\), which implies \((\mathbf{I} - \mathbf{P})(h_1 - h_2) = 0\).  
To solve \eqref{buweiyi0713}, we first consider the equation:  
\[
(1 + H)^3 \left[ (1 + H)^{-1} h \right]_{pp} + \gamma^2 (1 + H)^2 h_{qq} = 0,
\]  
which can be solved using separation of variables. We then apply the method of undetermined coefficients to obtain a particular solution of \eqref{buweiyi0713}. Thus, we derive a solution \(h\) in the form:  
\begin{align}\label{yigejie0703}
h = \frac{h_1(p, q) + h_2(p, q)}{\hat{C}^2},
\end{align}  
where  $h_1(p, q) $ and $h_2(p, q)$ are given by
\begin{align*}
h_1(p, q) &= -\frac{3}{2} e^{\gamma (3p - 1)} + \frac{1}{2} e^{-\gamma (p + 5)} + e^{\gamma (p - 3)} \cos(2q), \\
h_2(p, q) &= 4\gamma p e^{\gamma (p - 3)} - 4 e^{\gamma (p - 3)} + 4\gamma p e^{\gamma (p - 1)} + e^{\gamma (p - 1)} + 4\gamma p e^{\gamma (p + 1)} + 4 e^{\gamma (p + 1)} \\
&\quad + \frac{\gamma^3 p e^{\gamma (p - 1)}}{p_0^2} - \frac{\gamma^3 p e^{\gamma (p + 1)}}{2p_0^2} - \frac{\gamma^3 p e^{\gamma (p + 5)}}{2p_0^2} - \frac{\gamma^2 e^{\gamma (p - 1)}}{2p_0^2} + \frac{\gamma^2 e^{\gamma (p + 1)}}{2p_0^2} \\
&\quad + \frac{\gamma^2 e^{\gamma (p + 3)}}{2p_0^2} - \frac{\gamma^2 e^{\gamma (p + 5)}}{2p_0^2} + \bigg( \frac{11}{6} e^{-\gamma (p + 1)} - \frac{11}{9} e^{-\gamma (p + 3)} \\
&\quad - \frac{4}{3} e^{-\gamma (p + 5)} - \frac{5}{18} e^{3\gamma \left(p - \frac{5}{3}\right)} - \frac{\gamma^2 e^{-\gamma (p + 1)}}{2p_0^2} + \frac{\gamma^2 e^{3\gamma (p + 1)}}{2p_0^2} \bigg) \cos(2q).
\end{align*}

By a direct computation, one can conclude that $\frac{h_{1}\left(p,q\right)}{\hat{C}^2}$ satisfies the equation $\eqref{buweiyi0713}_{1}$. And $\frac{h_{2}\left(p,q\right)}{\hat{C}^2}$ solves the $\eqref{buweiyi0713}_{1}$ as the right side of $\eqref{buweiyi0713}_{1}$ is zero.  Subsequently, substitute \eqref{yigejie0703} into the following expression 
\begin{align*}
\frac{\langle\mathcal{F}_{ff}\left(\alpha_{c},0\right)\left[\hat{h}^{*},h\right], \mathbf{h}_{0}\rangle}{\langle\mathcal{F}_{\alpha f}\left(\alpha_{c},0\right)\hat{h}^{*}, \mathbf{h}_{0}\rangle},
\end{align*}
we obtain
\begin{align*}
&\begin{aligned}
    \mathcal{O}_{2}=&\frac{1}{36 \gamma ^2 \hat{C}^2 p_{0}^2 \left(e^{2 \gamma }-1\right)^2 e^{9 \gamma}}
    \bigg{[}60 p_{0}^4
    -63 \gamma ^4 e^{16 \gamma }+18 \gamma ^4 e^{18 \gamma }+86 p_{0}^4 e^{2 \gamma }
    \\
    &+e^{4 \gamma } \left(864 \gamma  p_{0}^4+530 p_{0}^4-25 \gamma ^2 p_{0}^2\right)
    \\&+e^{6 \gamma } \left(576 \gamma  p_{0}^4-2230 p_{0}^4+144 \gamma ^3 p_{0}^2+324 \gamma ^2 p_{0}^2\right)
    \\
    &+e^{8 \gamma } \left(27 \gamma ^4+576 \gamma   p_{0}^4+114 p_{0}^4-72 \gamma ^3  p_{0}^2-930 \gamma ^2  p_{0}^2\right)\\&+e^{10 \gamma } \left(1440 p_{0}^4-72 \gamma ^4+448 \gamma ^2  p_{0}^2\right)
    \\
    &+e^{12 \gamma } \left(36 \gamma ^4-72 \gamma ^3 p_{0}^2+507\gamma ^2 p_{0}^2\right)+e^{14 \gamma } \left(54 \gamma ^4-324 \gamma ^2 p_{0}^2\right)
    \bigg{]}.
\end{aligned}
\end{align*}
Thus, we obtain the exact expression of $\mathcal{O}$ as follows
\begin{align}\label{fuhao0715}
\begin{aligned}
\mathcal{O}=&
\frac{1}{36 \gamma ^2 \hat{C}^2 p_{0}^2 \left(e^{2 \gamma }-1\right)^2 e^{9 \gamma}}\bigg{[}
60 p_{0}^4+230 p_{0}^4 e^{2 \gamma }\\
&+e^{4 \gamma } \left(962  p_{0}^4-79 \gamma^2  p_{0}^2+1152\gamma p_{0}^4\right)
\\
&+e^{6 \gamma } \left(576 \gamma  p_{0}^4-3814 p_{0}^4+144 \gamma ^3 p_{0}^2+648 \gamma ^2 p_{0}^2\right)\\&+e^{8 \gamma } (27 \gamma ^4+576 \gamma p_{0}^4+1122 p_{0}^4-72 \gamma ^3 p_{0}^2
-1578 \gamma ^2 p_{0}^2)\\
&+e^{10 \gamma } \left(1440 p_{0}^4-72 \gamma ^4+988 \gamma ^2 p_{0}^2\right)+e^{14 \gamma } \left(54 \gamma ^4-324 \gamma ^2 p_{0}^2\right)
\\
&+e^{12 \gamma } \left(36 \gamma ^4-72 \gamma ^3 p_{0}^2+345 \gamma ^2 p_{0}^2\right)-63 \gamma ^4 e^{16 \gamma }+18 \gamma ^4 e^{18 \gamma }
\bigg{]}.
\end{aligned}
\end{align}
\end{proof}
\section{Large-amplitude traveling waves}\label{formu0520}
In the context of analyzing large bifurcation solutions, topological degree theory plays a pivotal role. When dealing with large bifurcation solutions, the complexity of the system often increases significantly. Topological degree theory provides a robust framework to handle this complexity effectively. The theory allows us to understand the global structure of the solution set. This section aims to extend the classical degree theory to address our specific problem \eqref{operator-equation}.

\subsection{Some useful lemmas}

\begin{lemma}\label{lemma31}
Let \(K\) be a compact set in \(Y\), and let \(B\) be a closed and bounded set in \(\overline{\mathcal{O}_{\delta}}\). Then, \(\mathcal{G}^{-1}(K) \cap B\) is compact in \((\alpha_s, +\infty) \times X\).
\end{lemma}
\begin{proof}
First, observe that for each \(\alpha \in (\alpha_s, +\infty)\), the coefficients of the second-order terms in \(\mathcal{G}_1(\alpha, h)\) satisfy  
\[
\left( (h + 1)^2 + (h_q)^2 \right) (h_p)^2 - (h_p h_q)^2 = (h + 1)^2 (h_p)^2 > \delta^2,
\]  
which implies that the nonlinear operator \(h \mapsto \mathcal{G}_1(\alpha, h)\) is uniformly elliptic in \(\mathcal{O}_{\delta}\).  

Second, note that for each \(\alpha\), by the definition of \(\mathcal{O}_{\delta}\), the coefficient of \(h_p\) in \(\mathcal{G}_2(\alpha, h)\) satisfies  
\[
(h + 1) h_p \left( (h + 1)^3 + 2\lambda(\alpha, \gamma)(h + 1) - 2\alpha (h + 1)^2 \right) > \delta^2,
\]  
which ensures that the boundary operator \(h \mapsto \mathcal{G}_2(\alpha, h)\) is uniformly oblique on \(T\) in the sense that it is bounded away from being tangential.

Let \(\{(f_k, g_k)\}\) be a convergent sequence in \(Y_1 \times Y_2 = C_{\text{per}}^{1+s}(\overline{\Omega}) \times C_{\text{per}}^{2+s}(T)\). Suppose also that  
\[
\begin{cases}
\mathcal{G}_1(\alpha_k, h_k) = f_k \quad \text{in } \overline{\Omega}, \\
\mathcal{G}_2(\alpha_k, h_k) = g_k \quad \text{on } T,
\end{cases}
\]  
where \((\alpha_k, h_k) \in \overline{\mathcal{O}_{\delta}}\), with \(h_k\) bounded in \(C_{\text{per}}^{3+s}(\overline{\Omega})\) and \(\alpha_k\) bounded in \((\alpha_s, +\infty)\). We now show that there exists a subsequence of \((\alpha_k, h_k)\) that converges in \((\alpha_s, +\infty) \times X\).  

Let \(\eta_j = \partial_q h_j\). Differentiating the system with respect to \(q\) yields  
\[
\begin{cases}
\partial_q \mathcal{G}_1(\alpha_j, h_j) = \partial_q f_j \quad \text{in } \overline{\Omega}, \\
\partial_q \mathcal{G}_2(\alpha_j, h_j) = \partial_q g_j \quad \text{on } T,
\end{cases}
\]  
which implies that \(\eta_j\) satisfies the following system:
\begin{align}\label{eq-eta}
\begin{cases}
\begin{aligned}
&(\partial_ph_j)^2(\partial_q^2\eta_j) 
+\left((h_j+1)^2
+(\partial_qh_j)^2\right)(\partial_p^2\eta_j)-
2(\partial_qh_j)(\partial_ph_j)(\partial_p\partial_q\eta_j)
\\&=F_1\left(
\partial_ph_j,\partial_qh_j,
\partial_p\partial_qh_j,
\partial_p^2h_j,
\partial_q^2h_j\right)+
\partial_qf_j,\quad (q,p)\in \Omega,
\end{aligned}\\
\begin{aligned}
&2(h_j+1)\partial_ph_j
\left(
2\lambda(\alpha_j,\gamma)(h_j+1)+(h_j+1)^3-2\alpha_j (h_j+1)^2\right)\partial_p\eta_j
-2p_{0}^2\partial_qh_j\partial_q\eta_j
\\&=F_2\left(
\partial_ph_j,\partial_qh_j,
\partial_p\partial_qh_j,
\partial_p^2h_j,
\partial_q^2h_j\right)+
\partial_qg_j,\quad p=0,
\end{aligned}\\
\eta_j=0,\quad p=-1,
\end{cases}
\end{align}
where $F_1\left(
\partial_ph_j,\partial_qh_j,
\partial_p\partial_qh_j,
\partial_p^2h_j,
\partial_q^2h_j\right)$
and $F_2\left(
\partial_ph_j,\partial_qh_j,
\partial_p\partial_qh_j,
\partial_p^2h_j,
\partial_q^2h_j\right)$ are given by \begin{align}\label{eq-eta-0520}
\begin{aligned}
&\begin{aligned}
&F_1\left(
\partial_ph_j,\partial_qh_j,
\partial_p\partial_qh_j,
\partial_p^2h_j,
\partial_q^2h_j\right)\\&=-2(\partial_q^2h_j)( \partial_ph_j)( \partial_p\partial_qh_j)
+\partial_qh_j(\partial_ph_j)^2\\
&\quad+2(h_j+1) (\partial_ph_j )( \partial_p\partial_qh_j)
-2(h_j+1)\partial_qh_j(\partial_p^2h_j)
\\&\quad+2(\partial_q^2h_j)(\partial_ph_j)(\partial_p\partial_qh_j)
+2(\partial_qh_j)(\partial_p\partial_qh_j)^2
\\&\quad-2(\partial_p^2h_j)(\partial_qh_j)(\partial_q^2h_j),
\end{aligned}\\
&\begin{aligned}
&F_2\left(
\partial_ph_j,\partial_qh_j,
\partial_p\partial_qh_j,
\partial_p^2h_j,
\partial_q^2h_j\right)\\&=
-2(h_j+1)\partial_qh_j
\left(\partial_ph_j\right)^2\left(
2\lambda(\alpha_{j},\gamma)+(h_j+1)^2-2\alpha_{j} (h_j+1)\right)
\\&\quad-(h_j+1)^2\left(\partial_ph_j\right)^2
\left(2(h_j+1)\partial_q h_j-2\alpha_{j} \partial_qh_j\right)
+2p_{0}^2(h_j+1)\partial_qh_j
\end{aligned}
\end{aligned}
\end{align}
which are cubic polynomial expressions. Note that  
\[
\left\{ F_1\left( \partial_p h_j, \partial_q h_j, \partial_p \partial_q h_j, \partial_p^2 h_j, \partial_q^2 h_j \right) \right\}_{j=1}^{\infty}
\]  
is uniformly bounded in \(C_{\text{per}}^{1+s}(\overline{\Omega})\), and \(\{\partial_q f_j\}_{j=1}^{\infty}\) is a Cauchy sequence in \(C_{\text{per}}^{s}(\overline{\Omega})\). Therefore,  
\[
\left\{ F_1\left( \partial_p h_j, \partial_q h_j, \partial_p \partial_q h_j, \partial_p^2 h_j, \partial_q^2 h_j \right) + \partial_q f_j \right\}_{j=1}^{\infty}
\]  
is compact in \(C_{\text{per}}^{s}(\overline{\Omega})\). Similarly,  
\[
\left\{ F_2\left( \partial_p h_j, \partial_q h_j, \partial_p \partial_q h_j, \partial_p^2 h_j, \partial_q^2 h_j \right) + \partial_q g_j \right\}_{j=1}^{\infty}
\]  
is compact in \(C_{\text{per}}^{1+s}(T)\).

We may assume that the following two sequences  
\begin{align*}
&\left\{ F_1\left( \partial_p h_j, \partial_q h_j, \partial_p \partial_q h_j, \partial_p^2 h_j, \partial_q^2 h_j \right) + \partial_q f_j \right\}_{j=1}^{\infty}, \\
&\left\{ F_2\left( \partial_p h_j, \partial_q h_j, \partial_p \partial_q h_j, \partial_p^2 h_j, \partial_q^2 h_j \right) + \partial_q g_j \right\}_{j=1}^{\infty}
\end{align*}  
are convergent in \(C_{\text{per}}^{s}(\overline{\Omega})\) and \(C_{\text{per}}^{1+s}(T)\), respectively.  

Taking differences, we deduce from equation \eqref{eq-eta} that \(\eta_{jk} := \eta_j - \eta_k\) satisfies  
\begin{align}\label{eq-eta-2}
\begin{cases}
\begin{aligned}
&(\partial_p h_j)^2 (\partial_q^2 \eta_{jk}) + \left( (h_j + 1)^2 + (\partial_q h_j)^2 \right) (\partial_p^2 \eta_{jk}) \\
&\quad - 2 (\partial_q h_j)(\partial_p h_j) (\partial_p \partial_q \eta_{jk}) = R_{1jk}, \quad (q, p) \in \Omega,
\end{aligned} \\
\begin{aligned}
&2 (h_j + 1) \partial_p h_j \left( 2\lambda(\alpha_j, \gamma)(h_j + 1) + (h_j + 1)^3 - 2\alpha_j (h_j + 1)^2 \right) \partial_p \eta_{jk} \\
&\quad - 2p_0^2 \partial_q h_j \partial_q \eta_{jk} = R_{2jk}, \quad p = 0,
\end{aligned} \\
\eta_{jk} = 0, \quad p = -1,
\end{cases}
\end{align}  
where \(R_{1jk}\) and \(R_{2jk}\) are given by  
\begin{align*}
\begin{aligned}
&R_{1jk} = F_1\left( \partial_p h_j, \partial_q h_j, \partial_p \partial_q h_j, \partial_p^2 h_j, \partial_q^2 h_j \right) \\
&\quad - F_1\left( \partial_p h_k, \partial_q h_k, \partial_p \partial_q h_k, \partial_p^2 h_k, \partial_q^2 h_k \right) + \partial_q (f_j - f_k) \\
&\quad + \left[ (\partial_p h_k)^2 - (\partial_p h_j)^2 \right] \partial_q^2 \eta_k + \left[ (1 + h_k)^2 - (1 + h_j)^2 + (\partial_q h_k)^2 - (\partial_q h_j)^2 \right] \partial_p^2 \eta_k \\
&\quad - 2 \left( \partial_q h_k \partial_p h_k - \partial_q h_j \partial_p h_j \right) \partial_p \partial_q \eta_k, \\
&R_{2jk} = F_2\left( \partial_p h_j, \partial_q h_j, \partial_p \partial_q h_j, \partial_p^2 h_j, \partial_q^2 h_j \right) \\
&\quad - F_2\left( \partial_p h_k, \partial_q h_k, \partial_p \partial_q h_k, \partial_p^2 h_k, \partial_q^2 h_k \right) \\
&\quad + 2 \left\{ (1 + h_k) \partial_p h_k \left[ 2\lambda(\alpha_k, \gamma)(1 + h_k) + (1 + h_k)^3 - 2\alpha_k (1 + h_k)^2 \right] \right. \\
&\quad \left. - (1 + h_j) \partial_p h_j \left[ 2\lambda(\alpha_j, \gamma)(1 + h_j) + (1 + h_j)^3 - 2\alpha_j (1 + h_j)^2 \right] \right\} \partial_p \eta_k \\
&\quad - 2p_0^2 (\partial_q h_k - \partial_q h_j) \partial_q \eta_k + \partial_q (g_j - g_k),
\end{aligned}
\end{align*}  
and which satisfy  
\[
\| R_{1jk} \|_{C_{\text{per}}^{s}(\overline{\Omega})} \to 0, \quad \| R_{2jk} \|_{C_{\text{per}}^{1+s}(T)} \to 0.
\]

Hence, we apply the mixed-boundary condition Schauder estimates \cite{Agmon1964} to get 
\[
\| \eta_{jk} \|_{C_{\text{per}}^{2+s}(\overline{\Omega})} \to 0 \quad \text{as} \quad j, k \to \infty.
\]  
Therefore, all third derivatives of \(h_j\), except possibly \(\partial_p^3 h_j\), are Cauchy sequences in \(C_{\text{per}}^{s}(\overline{\Omega})\). Since \(\partial_p^3 h_j\) can be expressed in terms of the other derivatives of \(h_j\) of order at most two, it follows that \(\partial_p^3 h_j\) is also a Cauchy sequence in \(C_{\text{per}}^{s}(\overline{\Omega})\).  We thus conclude that there exists a subsequence of \((\alpha_k, h_k)\) that converges in \((\alpha_s, +\infty) \times X\).
\end{proof}

\begin{lemma}\label{lemma32}
For each \((\alpha, h) \in \mathcal{O}_{\delta}\), \(\mathcal{G}_h(\alpha, h)\) is a Fredholm operator with index zero from \(X\) to \(Y\).
\end{lemma}
\begin{proof}
To derive the linearized operator, we differentiate the nonlinear operator \(\mathcal{G}(\alpha, h)\) defined in \eqref{operator-equation} with respect to \(h\). This yields the following expressions for the components of the Fréchet derivative \(\mathcal{G}_h(\alpha, h) = (\mathcal{G}_{1h}(\alpha, h), \mathcal{G}_{2h}(\alpha, h))\):
\begin{align*}
&\begin{aligned}
\mathcal{G}_{1h}(\alpha, h) = {} & (h_p)^2 \partial^2_q - 2h_q h_p \partial^2_{qp} + \left( (h + 1)^2 + (h_q)^2 \right) \partial^2_p \\
& + 2h_{qq} h_p \partial_p - (h_p)^2 - 2(h + 1) h_p \partial_p + 2h_{pp} (h + 1) \\
& - 2h_q h_{qp} \partial_p - 2h_p h_{qp} \partial_q + 2h_{pp} h_q \partial_q,
\end{aligned} \\
&\begin{aligned}
\mathcal{G}_{2h}(\alpha, h) = {} & 2(h + 1) h_p \left( 2\lambda(\alpha, \gamma)(h + 1) + (h + 1)^3 - 2\alpha (h + 1)^2 \right) \partial_p \\
& - 2p_0^2 h_q \partial_q + 2(h_p)^2 \left( 2\lambda(\alpha, \gamma)(h + 1) + (h + 1)^3 - 2\alpha (h + 1)^2 \right) \\
& + 2(h + 1)^2 (h_p)^2 (h + 1 - \alpha) - 2p_0^2 (h + 1).
\end{aligned}
\end{align*}

The coefficients of the second-order terms in \(\mathcal{G}_{1h}(\alpha, h)\) satisfy  
\[
\left( (h + 1)^2 + (h_q)^2 \right) (h_p)^2 - (h_p h_q)^2 = (h + 1)^2 (h_p)^2 > \delta^2,
\]  
which implies that the linearized operator \(\mathcal{G}_{1h}(\alpha, h)\) is uniformly elliptic in \(\mathcal{O}_{\delta}\). 
Furthermore, for each \(\alpha\), the definition of \(\mathcal{O}_{\delta}\) ensures that the coefficient of the \(\partial_p\) term in \(\mathcal{G}_{2h}(\alpha, h)\) satisfies  
\[
2(h + 1) h_p \left( 2\lambda(\alpha, \gamma)(h + 1) + (h + 1)^3 - 2\alpha (h + 1)^2 \right) > 2\delta^2.
\]  
This implies that \(\mathcal{G}_{2h}(\alpha, h)\) is uniformly oblique on \(T\), as it is bounded away from being tangential.

Let \(\phi \in C_{\text{per}}^{3+s}(\overline{\Omega})\). Taking the \(q\)-derivatives of \(\mathcal{G}_{1h}(\alpha, h)\phi\) and \(\mathcal{G}_{2h}(\alpha, h)\phi\), we obtain:
\begin{align}
\begin{aligned}
\mathcal{G}_{1h}(\alpha, h) \partial_q \phi &= \partial_q \left( \mathcal{G}_{1h}(\alpha, h) \phi \right) - \left( \partial_q \mathcal{G}_{1h}(\alpha, h) \right) \phi, \\
\mathcal{G}_{2h}(\alpha, h) \partial_q \phi &= \partial_q \left( \mathcal{G}_{2h}(\alpha, h) \phi \right) - \left( \partial_q \mathcal{G}_{2h}(\alpha, h) \right) \phi.
\end{aligned}
\end{align}

Based on classical Schauder estimates \cite{Gilbarg2001}, there exists a constant \(C > 0\) such that:
\[
\| \partial_q \phi \|_{C_{\text{per}}^{2+s}(\overline{\Omega})} \leq C \left( \| \phi \|_{C_{\text{per}}^{2+s}(\overline{\Omega})} + \| \left( \partial_q \mathcal{G}_{1h}(\alpha, h) \right) \phi \|_{C_{\text{per}}^{s}(\overline{\Omega})} + \| \left( \partial_q \mathcal{G}_{2h}(\alpha, h) \right) \phi \|_{C_{\text{per}}^{s}(T)} \right).
\]

Since \(\partial_p^2 \phi\) can be expressed in terms of lower-order derivatives (up to first order), it follows that \(\| \partial_q^3 \phi \|_{C_{\text{per}}^{s}(\overline{\Omega})}\) admits a similar estimate. Consequently, we have:
\[
\| \phi \|_{C_{\text{per}}^{3+s}(\overline{\Omega})} \leq C \left( \| \phi \|_{C_{\text{per}}^{2+s}(\overline{\Omega})} + \| \left( \partial_q \mathcal{G}_{1h}(\alpha, h) \right) \phi \|_{C_{\text{per}}^{s}(\overline{\Omega})} + \| \left( \partial_q \mathcal{G}_{2h}(\alpha, h) \right) \phi \|_{C_{\text{per}}^{s}(T)} \right).
\]

From the classical Schauder estimate for \(\phi\):
\[
\| \phi \|_{C_{\text{per}}^{2+s}(\overline{\Omega})} \leq C \left( \| \phi \|_{C_{\text{per}}^{s}(\overline{\Omega})} + \| \mathcal{G}_{1h}(\alpha, h) \phi \|_{C_{\text{per}}^{s}(\overline{\Omega})} + \| \mathcal{G}_{2h}(\alpha, h) \phi \|_{C_{\text{per}}^{1+s}(T)} \right),
\]
we deduce:
\[
\| \phi \|_X \leq C \left( \| \phi \|_{Y_1} + \| \mathcal{G}_{1h}(\alpha, h) \phi \|_{Y_1} + \| \mathcal{G}_{2h}(\alpha, h) \phi \|_{Y_2} \right).
\]

This inequality immediately implies that \(\mathcal{G}_h(\alpha, h)\) has a finite-dimensional null space and a closed range. We have already shown that \(\mathcal{G}_h(\alpha_c, H) = \mathcal{F}_f(\alpha_c, 0)\) has index zero, as its null space has dimension 1 and its range has codimension 1. Since the Fredholm index is a continuous function of \((\alpha, h)\) and \(\mathcal{O}_\delta\) is connected, the index of \(\mathcal{G}_h(\alpha, h)\) vanishes for all \((\alpha, h) \in \mathcal{O}_\delta\).\end{proof}

Our next preliminary result concerns the spectral properties of the linear operator \(\mathcal{G}_h(\alpha, h)\) for \((\alpha, h) \in \mathcal{O}_\delta\). To facilitate the discussion, we introduce the following notations:  
\[
\mathcal{P} := \mathcal{G}_{1h}(\alpha, h), \quad \mathcal{Q} := \mathcal{G}_{2h}(\alpha, h).
\]  
Define the domain \(D\) for the operator $\mathcal{P}$ as  
\begin{align}
D = \{ \phi \in X \mid \mathcal{Q} \phi = 0 \}.
\end{align}  
The spectrum of \(\mathcal{P}\) is defined by  
\begin{align}
\Sigma(\alpha, h) = \{ \lambda \in \mathbb{C} \mid \mathcal{P} - \lambda I : D \to Y_1 \text{ is not an isomorphism} \},
\end{align}  
where \(\mathcal{P}\), \(D\), and \(Y_1\) are complexified in the natural way.

\begin{lemma}\label{lemma33}(Spectral properties).
\begin{description}
\item[(1)]  For all \(\delta \in (0, \gamma)\), there exist constants \(c_1, c_2 > 0\) such that for all \((\alpha, h) \in \mathcal{O}_{\delta}\) with \(\alpha + \| h \|_X \leq M\), for all \(\phi \in X\), and for all real \(\lambda \geq c_2\), the following estimate holds:
\begin{align}
c_1 \| \phi \|_X \leq \lambda^{\frac{s}{2}} \| (\mathcal{P} - \lambda I) \phi \|_{Y_1} + \lambda^{\frac{s + 1}{2}} \| \mathcal{Q} \phi \|_{Y_2}.
\end{align}
\item[(2)] The spectrum \(\Sigma(\alpha, h)\) consists entirely of eigenvalues of finite multiplicity with no finite accumulation points. Furthermore, there exists a neighborhood \(U\) of \([\alpha_s, +\infty)\) in the complex plane such that \(\Sigma(\alpha, h) \cap U\) is a finite set.
\item[(3)] For all \((\alpha, h) \in \mathcal{O}_{\delta}\), the boundary operator \(\mathcal{Q} : X \to Y_2\) is surjective.
\end{description}
\end{lemma}
\begin{proof}
The second conclusion (2) follows from (1). Likewise, a similar argument establishes that \((\mathcal{P} - \lambda I, \mathcal{Q})\) is Fredholm with index 0. By (1), it has a trivial null space for sufficiently large \(|\lambda|\), and hence is injective and surjective. This implies (2).  
To prove (3), note that for some \(\lambda > c_2\), the operator \((\mathcal{P} - \lambda I, \mathcal{Q})\) is surjective onto \(Y\). It follows that \(\mathcal{Q} : X \to Y_2\) is surjective.  

We now follow the approach in \cite{Agmon1962} to establish (1). The novelty lies in (1), where additional work is required before applying elliptic estimates. Let \(\phi \in X\) and \(\lambda \in \mathbb{C}\) be given, and set \(\theta = \arg \lambda\), with \(|\theta| \leq \frac{\pi}{2} + \epsilon\) for some \(\epsilon > 0\).  

Consider the operator \(\mathcal{H} = \mathcal{P} + e^{i\theta} \partial_t^2\) on \(\mathbb{R} \times \Omega\), where \(\mathcal{P} = \mathcal{G}_{1h}(\alpha, h)\). One can verify that \(\mathcal{H}\) is elliptic, and the boundary condition on the top is oblique. Let \(\eta : \mathbb{R} \to \mathbb{R}\) be a cutoff function compactly supported in \(I = (-1, 1)\). Define  
\[
r(t) = e^{i|\lambda|^{\frac{1}{2}}} \eta(t), \quad \psi(t, q, p) = r(t) \phi(q, p).
\]

Applying the Schauder estimates on \(\mathbb{R} \times \Omega\) with the boundary operator \(\mathcal{Q}\) to \(\psi(t, q, p)\), we deduce the inequality:
\begin{align}\label{keyone}
\norm{\psi}_{C^{3+s}(I \times \overline{\Omega})}
\leq C \left(
\norm{\psi}_{C^{s}(I \times \overline{\Omega})} +
\norm{\mathcal{H} \psi}_{C^{1+s}(I \times \overline{\Omega})} +
\norm{\mathcal{Q} \psi}_{C^{2+s}(I \times T)}
\right).
\end{align}

Note that there exist constants \(C_1, C_2 > 0\) such that \(r(t)\) satisfies:
\[
C_1 |\lambda|^{\frac{s}{2}} \leq \norm{r(t)}_{C^{s}(\mathbb{R})} \leq C_2 |\lambda|^{\frac{s}{2}}, \quad
C_1 |\lambda|^{\frac{1 + s}{2}} \leq \norm{r(t)}_{C^{1+s}(\mathbb{R})} \leq C_2 |\lambda|^{\frac{1 + s}{2}}.
\]
For \(|\lambda|\) sufficiently large, this help us unpack \eqref{keyone} to derive the following estimates:
\begin{align}\label{keytwo}
\begin{aligned}
\norm{\mathcal{H} \psi}_{C^{1+s}(I \times \overline{\Omega})}
&= \norm{
r(t) \mathcal{P} \phi - \lambda r(t) \phi +
\left( 2i |\lambda|^{\frac{1}{2}} \eta' + \eta'' \right) e^{i\theta} e^{i |\lambda|^{\frac{1}{2}} t} \phi
}_{C^{1+s}(I \times \overline{\Omega})} \\
&\leq C |\lambda|^{\frac{1 + s}{2}} \norm{ (\mathcal{P} - \lambda I) \phi }_{C^{1+s}(\overline{\Omega})} 
+ C |\lambda|^{\frac{1 + s}{2}} \norm{\phi}_{C^{1+s}(\overline{\Omega})},
\end{aligned}
\end{align}

In a similar manner, analyzing the boundary terms, we find that for some \(C > 0\) independent of \(\phi\) and for \(|\lambda|\) sufficiently large,
\begin{align}\label{keythree}
\norm{\mathcal{Q} \psi}_{C^{2+s}(I \times T)} \leq C |\lambda|^{\frac{2 + s}{2}} \norm{\phi}_{C^{2+s}(T)}.
\end{align}
Combining these estimates, we obtain:
\begin{align}\label{keyfour}
\begin{aligned}
\norm{\psi}_{C^{3+s}(I \times \overline{\Omega})} \leq{} & C \norm{\psi}_{C^{s}(I \times \overline{\Omega})} 
+ C |\lambda|^{\frac{1 + s}{2}} \left( \norm{ (\mathcal{P} - \lambda I) \phi }_{C^{1+s}(\overline{\Omega})} + \norm{\phi}_{C^{1+s}(\overline{\Omega})} \right) \\
& + C |\lambda|^{\frac{2 + s}{2}} \norm{\phi}_{C^{2+s}(T)}.
\end{aligned}
\end{align}

Note that for \(|\lambda|\) sufficiently large, the following bound holds:
\[
C \norm{\psi}_{C^{s}(I \times \overline{\Omega})} + C |\lambda|^{\frac{1 + s}{2}} \norm{\phi}_{C^{1+s}(\overline{\Omega})} \leq \frac{1}{2} \norm{\psi}_{C^{3+s}(I \times \overline{\Omega})},
\]
which implies, after rearrangement,
\begin{align}\label{keyfive}
\norm{\psi}_{C^{3+s}(I \times \overline{\Omega})} \leq C |\lambda|^{\frac{1 + s}{2}} \norm{ (\mathcal{P} - \lambda I) \phi }_{C^{1+s}(\overline{\Omega})} + C |\lambda|^{\frac{2 + s}{2}} \norm{\mathcal{Q} \phi}_{C^{2+s}(T)}.
\end{align}
\end{proof}
\subsection{The degree}\label{degree-theory}

The nonlinear boundary condition renders the classical Leray–Schauder degree inapplicable to our problem. Instead, we employ the variant developed by Healey and Simpson \cite{Healey1998}. Specifically, let \(W\) be an open bounded subset of \(X\), and fix \(\alpha \in [\alpha_s, +\infty)\). The results established in \autoref{lemma31}–\autoref{lemma33} allow us to apply the generalization of the Leray–Schauder degree introduced in that work, which we now briefly summarize.

Let \(W\) be an open bounded subset of \(X\), and let \(f \in Y \setminus \mathcal{G}(\alpha, \partial W)\) be a regular value of \(\mathcal{G}(\alpha, \cdot)\). The Healey–Simpson degree of \(\mathcal{G}(\alpha, \cdot)\) at \(f\) with respect to \(W\) is defined by  
\[
\text{deg}(\mathcal{G}(\alpha, \cdot), W, f) = \sum_{w \in \mathcal{G}^{-1}(f)} (-1)^{n(w)},
\]  
where \(n(w)\) denotes the number of positive real eigenvalues \(\sigma\) (counted with multiplicity) of the linear operator \(\mathcal{G}_{1w}(\alpha, w)\) under the boundary condition \(\mathcal{G}_{2w}(\alpha, w)h = 0\) on \(T\). That is, \(n(w)\) is determined by the eigenvalue problem:  
\[
\begin{cases}
\mathcal{G}_{1w}(\alpha, w)h = \sigma h, \\
\mathcal{G}_{2w}(\alpha, w)h = 0.
\end{cases}
\]

By \autoref{lemma33}, \(n(h)\) is finite. Moreover, the properness of \(\mathcal{G}(\alpha, \cdot)\) established in \autoref{lemma31} implies that \(\mathcal{G}^{-1}(f) \cap W\) is a finite set. Thus, the degree is well-defined at regular values. This definition is extended to critical values in the standard way via the Sard–Smale theorem (which is applicable by \autoref{lemma32}).  
Our interest in degree theory stems from the fact that the degree is invariant under any homotopy that respects the boundaries of \(W\) in the following sense. Let \(Q \subset [0, 1] \times W\) be open. We define  $Q_t$ and its boundary set as
\[
Q_t := \{ w \in W : (t, w) \in Q \}, \quad \partial Q_t := \{ w \in W : (t, w) \in \partial Q \}.
\]  

\begin{lemma}[Homotopy invariance]  
Suppose \(Q \subset [0, 1] \times W\) is open. If \(\mathcal{H} \in C^2(Q; Y)\) is proper and, for each \(t \in [0, 1]\), \(\mathcal{H}(t, \cdot)\) satisfies the conclusions of  \autoref{lemma31}–\autoref{lemma33}, then  
\[
\text{deg}(\mathcal{H}(0, \cdot), Q_0, f) = \text{deg}(\mathcal{H}(1, \cdot), Q_1, f),
\]  
provided \(f \notin \mathcal{H}(t, \partial Q_t)\) for all \(t \in [0, 1]\).  
\end{lemma}
\begin{proof}
The proof is standard and can be found in various references, for example \cite{Healey1998}.
\end{proof}

Let us define the open ball centered at \(H\) with radius \(r\) as  
\[
B(H; r) = \{ w \in X \mid \| w - H \|_X < r \}.
\]
Based on the definition of the topological degree above, we have the following result:  

\begin{lemma}\label{egien}  
For sufficiently small \(r > 0\), the following holds:
\[
\text{deg}(\mathcal{G}(\alpha, \cdot), B(H, r), 0) = (-1)^{m(\alpha)},
\]  
where \(m(\alpha)\) denotes the number of positive real eigenvalues (counted with multiplicity) of the linearized operator \(\mathcal{F}_{f}(\alpha, 0)\) under the boundary condition \(\mathcal{F}_{2f}(\alpha, 0)h = 0\) on \(T\).  
\end{lemma}

\subsection{Proof of \autoref{large-waves}}

\begin{proof}
Assume, for contradiction, that \(\mathcal{C}_{\delta}\) is bounded in \((\alpha_s, \infty) \times X\), does not intersect \(\{(\alpha, H)\}\) except at \(\alpha = \alpha_c\), and does not meet \(\partial \mathcal{O}_{\delta}\).  

Let \(\mathcal{Q}\) be a closed set in \((\alpha_s, \infty) \times X\) containing \(\mathcal{C}_{\delta}\). By \autoref{lemma31}, \(\mathcal{G}^{-1}(0) \cap \mathcal{Q}\) is compact, and since \(\mathcal{C}_{\delta}\) is a closed subset of this set, it follows that \(\mathcal{C}_{\delta}\) is compact in \((\alpha_s, \infty) \times X\).  

Note that \(\mathcal{G}(\alpha, H) = 0\) for all \(\alpha \geq \alpha_s\), and there exist \(\alpha_1, \alpha_2\) such that \(\alpha_1 < \alpha_c < \alpha_2\) and \((\alpha, H)\) is not a bifurcation point for \(\alpha \neq \alpha_c\). The isolation of the trivial solution implies the existence of a positive function \(r(\alpha)\) defined on \([\alpha_1, \alpha_2]\) and a constant \(r_0 > 0\) such that:  
\begin{align}
\begin{aligned}
&r(\alpha) \to 0 \quad \text{as} \quad \alpha \to \alpha_c, \quad r(\alpha) \leq r_0, \\
&\{(\alpha, h) \in [\alpha_s, +\infty) \times X \mid \| h - H \|_X < r(\alpha)\} \cap \mathcal{C}_{\text{loc}} = \emptyset,
\end{aligned}
\end{align}  
when \(\alpha_1\) and \(\alpha_2\) are sufficiently close to \(\alpha_c\).  

Define the distorted rectangle as follows:  
\[
\mathcal{R} = \bigcup_{\alpha \in (\alpha_1, \alpha_2)} \left\{ (\alpha, h) \mid \| h - H \|_X < r_0 \right\}.
\]
We claim that there exists an open and bounded set \(\mathcal{W}\) such that:
\begin{description}
\item[(1)] \(\mathcal{C}_{\delta} \subset \mathcal{W} \subset \mathcal{O}_{\delta}\);
\item[(2)] \(\mathcal{G} \neq 0\) on \(\partial \mathcal{W} \setminus \{(\alpha, H) \mid \alpha > \alpha_s\}\);
\item[(3)] \(\mathcal{W} \cap \{(\alpha, H) \mid \alpha > \alpha_s\} = \mathcal{R} \cap \{(\alpha, H) \mid \alpha > \alpha_s\} = \{(\alpha, H) \mid \alpha \in (\alpha_1, \alpha_2)\}\).
\end{description}
We choose a sufficiently small \(\epsilon > 0\) such that \(\mathcal{W}_1\) is an open \(\epsilon\)-neighborhood of \(\mathcal{C}_{\delta} \setminus \mathcal{R}\), and define
\[
\mathcal{W} = \mathcal{R} \cup \mathcal{W}_1.
\]
Properties (1) and (2) imply that \(\mathcal{C}_{\delta}\) does not intersect \(\partial \mathcal{W}\).  

Let us now define
\[
\mathcal{W}_{\alpha} = \{ h \in X \mid (\alpha, h) \in \mathcal{W} \}.
\]
and consider the degree of \(\mathcal{G}(\alpha, \cdot)\) over three distinct regimes:
\begin{description}
\item[(i)] By homotopy invariance and the fact that \(\mathcal{G} \neq 0\) on \(\partial \mathcal{W} \setminus \{(\alpha, H) \mid \alpha > \alpha_s\}\), the degree \(\text{deg}(\mathcal{G}(\alpha, \cdot), \mathcal{W}_{\alpha}, 0)\) is independent of \(\alpha\) for \(\alpha \in (\alpha_1, \alpha_2)\).

\item[(ii)] From \autoref{egien}, we have
\[
\text{deg}(\mathcal{G}(\alpha, \cdot), B(H, r(\alpha)), 0) = (-1)^{m(\alpha)},
\]
where \(m(\alpha)\) is the number of positive real eigenvalues (counted with multiplicity) of the linear operator \(\mathcal{G}_{1w}(\alpha, w)|_{w = H}\) subject to the boundary condition \(\mathcal{G}_{2w}(\alpha, w)|_{w = H} h = 0\) on \(T\). Since a simple eigenvalue crosses zero at \(\alpha = \alpha_c\), the value \(m(\alpha)\) changes by 1 as \(\alpha\) passes through \(\alpha_c\). Hence, \(\text{deg}(\mathcal{G}(\alpha, \cdot), B(H, r(\alpha)), 0)\) changes sign at \(\alpha = \alpha_c\).

\item[(iii)] Properties (2) and (3) imply that \(\mathcal{G}(\alpha, h) \neq 0\) for \(\alpha \neq \alpha_c\) and \(h \in \partial(\mathcal{W}_{\alpha} \setminus B(H, r(\alpha)))\). By homotopy invariance, \(\text{deg}(\mathcal{G}(\alpha, \cdot), \mathcal{W}_{\alpha} \setminus B(H, r(\alpha)), 0)\) is constant for \(\alpha < \alpha_c\) and \(\alpha > \alpha_c\). Moreover, there exists \(\alpha \in (\alpha_1, \alpha_2)\) such that \(\mathcal{W}_{\alpha} \setminus B(H, r(\alpha)) = \emptyset\), which implies
\[
\text{deg}(\mathcal{G}(\alpha, \cdot), \mathcal{W}_{\alpha} \setminus B(H, r(\alpha)), 0) = 0.
\]
\end{description}

By the additivity property of degree,
\[
\text{deg}(\mathcal{G}(\alpha, \cdot), \mathcal{W}_{\alpha} \setminus B(H, r(\alpha)), 0) + \text{deg}(\mathcal{G}(\alpha, \cdot), B(H, r(\alpha)), 0) = \text{deg}(\mathcal{G}(\alpha, \cdot), \mathcal{W}_{\alpha}, 0).
\]
From (iii), it follows that
\[
\text{deg}(\mathcal{G}(\alpha, \cdot), B(H, r(\alpha)), 0) = \text{deg}(\mathcal{G}(\alpha, \cdot), \mathcal{W}_{\alpha}, 0).
\]
The left-hand side changes sign as \(\alpha\) crosses \(\alpha_c\), while the right-hand side is independent of \(\alpha\). This contradiction completes the proof.
\end{proof}

\section{Numerical investigation}\label{shuzhi0805}
We numerically investigate the bifurcation types. At $\alpha=\alpha_c$, it gets from \eqref{biaodaishi0821} and \eqref{critical-alpha} that
\[
\begin{cases}
\alpha_c-\frac{2p_0^2}{\gamma^2 e^{\gamma} (e^{2\gamma} - 1)} = e^{\gamma} ,\\
2\gamma^2 e^{3\gamma}\alpha_c+p_0^2 = 
\gamma^2 e^{2\gamma}\left( 2\lambda +e^{2\gamma}\right) .
\end{cases}
\]
which means that $\alpha_c$ and $p_0$ are completely determined once $\gamma$ and $\lambda$ are given. Hence, we adopt \(\gamma\) and $\lambda$
as the primary parameters to analyze the bifurcation types.

We present a parameter-region plot in the \((\gamma, \lambda)\)-plane to illustrate the bifurcation type, as shown in \autoref{Critical0805}. Regions where the non-dimensional parameter \(\mathcal{O} > 0\) (defined in \eqref{fuhao0715}) correspond to supercritical bifurcation, while regions with \(\mathcal{O} < 0\) correspond to subcritical bifurcation.  
We select specific \((\gamma, \lambda)\) pairs to provide concrete examples and validate the theoretical predictions.

Subsequently, we provide two illustrative examples: one with \(\mathcal{O} > 0\) and the other with \(\mathcal{O} < 0\).

\begin{figure}[h]
  \centering
  \includegraphics[width=.6\textwidth,height=.5\textwidth]{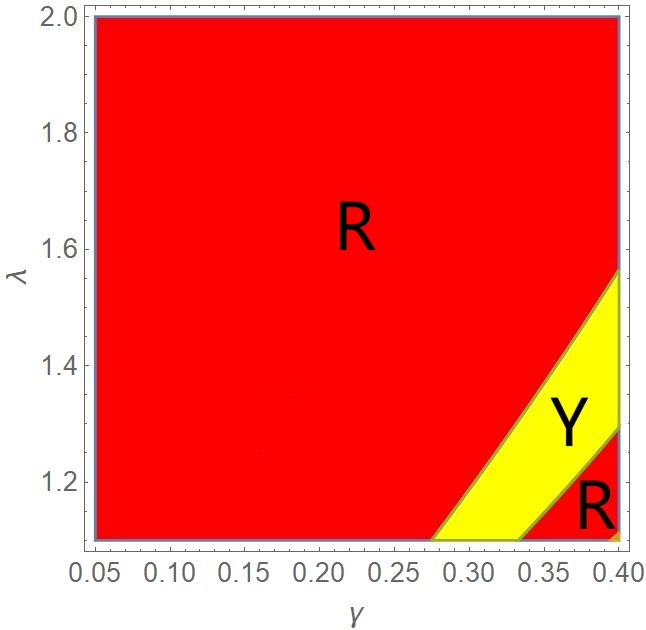}\\
  \caption{
  Red region (supercritical bifurcation) and yellow region (subcritical bifurcation).
  }\label{Critical0805}
\end{figure}

\begin{example}
Choose \((\gamma, \lambda) = (0.2, 1.4)\) from the red region in \autoref{Critical0805}. We obtain:
\[
\mathcal{O} = 0.218807 > 0, \quad p_0^2 = 0.00594402, \quad \gamma^2 e^{4\gamma} = 0.0890216 > p_0^2, \quad \alpha_c = 1.71615. 
\]
The system \eqref{new-equation} undergoes a supercritical bifurcation at the bifurcation point \((\alpha_c, H(p))\). Specifically, for \(\alpha\) in a neighborhood of \(\alpha_c\) with \(\alpha > \alpha_c\), a bifurcated solution to \eqref{new-equation} emerges.
    \begin{figure}[h]
  \centering
   \begin{tikzpicture}
    \draw[->] (-2,0) -- (4,0) node[right] {$\alpha$};
      \draw[dashed] (-2,0.5) -- (4,0.5) node[right] {$h=H$};
    \draw[->] (0,-2) -- (0,2) node[above] {$X$};
   \draw[dashed] (2,-1.8) -- (2,1.8);
     \node at (2,-1.5) [below right] {$\alpha=\alpha_c$};
    \node at (0,0) [below left] {$O$};
    \fill[color=blue] (2,0.5) circle (2pt);
    \draw[blue, thick] (2,0.5) .. controls (2.3,1.5) .. (3.5, 1.8);
    \draw[blue, thick] (2,0.5) .. controls (2.3,-0.5) .. (3.5,-0.8) ;
\end{tikzpicture}
  \caption{
 Supercritical bifurcation at $(\alpha_c,H)$ with $H=e^{0.2(p + 1)} - 1$.
  }
\end{figure}
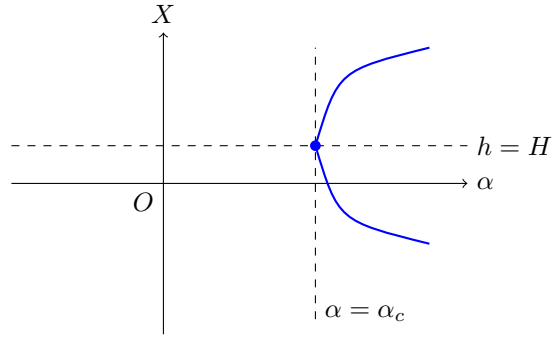
\end{example}

\begin{example}
Choose \((\gamma, \lambda) = (0.3, 1.15)\) from the yellow region in Figure \ref{Critical0805}. We obtain:
\[
\mathcal{O} = -0.150203 < 0, \quad p_0^2 = 0.00794367, \quad \gamma^2 e^{4\gamma} = 0.298811 > p_0^2, \quad \alpha_c = 1.50893.
\]
The system \eqref{new-equation} undergoes a subcritical bifurcation at the bifurcation point \((\alpha_c, H(p))\). Specifically, for \(\alpha\) in a neighborhood of \(\alpha_c\) with \(\alpha < \alpha_c\), a bifurcated solution to \eqref{new-equation} emerges.
      \begin{center}
          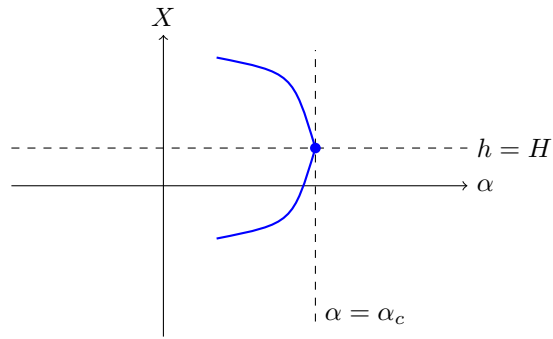
\begin{figure}[h]
  \centering
   \begin{tikzpicture}
    \draw[->] (-2,0) -- (4,0) node[right] {$\alpha$};
    \draw[->] (0,-2) -- (0,2) node[above] {$X$};
        \draw[dashed] (-2,0.5) -- (4,0.5) node[right] {$h=H$};
\draw[dashed] (2,-1.8) -- (2,1.8);
     \node at (2,-1.5) [below right] {$\alpha=\alpha_c$};
    \fill[color=blue] (2,0.5) circle (2pt);
   
\draw[blue, thick] (2,0.5) .. controls (1.7,1.5) .. (0.7,1.7);
\draw[blue, thick] (2,0.5) .. controls (1.7,-0.5) .. (0.7,-0.7);
\end{tikzpicture}
  \caption{
 Subcritical bifurcation at $(\alpha_c,H)$ with $H=e^{0.2(p + 1)} - 1$.  }
\end{figure}
\end{center}
\end{example}

 \section*{Acknowledgement}
L. Li was supported by the Young Scientists Fund of the 
National Natural Science Foundation of China (No. 12301131, 11901408).
Q. Wang was supported by the Natural Science Foundation of Sichuan Province (Grant No.2025ZNSFSC0072). 
\itemsep=0pt


\end{document}